\documentclass{article}%
\usepackage{amsmath}
\usepackage{amsfonts}
\usepackage{amssymb}
\usepackage{graphicx}
\usepackage{xcolor}%
\providecommand{\U}[1]{\protect\rule{.1in}{.1in}}
\newtheorem{theorem}{Theorem}

\newtheorem{definition}[theorem]{Definition}

\newtheorem{lemma}[theorem]{Lemma}

\newtheorem{problem}[theorem]{Problem}
\newtheorem{proposition}[theorem]{Proposition}
\newtheorem{remark}[theorem]{Remark}

\newenvironment{proof}[1][Proof]{\noindent\textbf{#1.} }{\ \rule{0.5em}{0.5em}}

\begin{document}

\title{On singular Galerkin discretizations for three models in high-frequency scattering}
\author{T. Chaumont-Frelet\thanks{(theophile.chaumont@inria.fr), Inria, Univ. Lille,
CNRS, UMR 8524 - Laboratoire Paul Painlev\'{e}, 59000 Lille, France}
\and S. Sauter\thanks{(stas@math.uzh.ch), Institut f\"{u}r Mathematik,
Universit\"{a}t Z\"{u}rich, Winterthurerstr 190, CH-8057 Z\"{u}rich,
Switzerland}}
\maketitle

\begin{abstract}
We consider three common mathematical models for time-harmonic high frequency
scattering: the Helmholtz equation in two and three spatial dimensions, a
transverse magnetic problem in two-dimensions, and Maxwell's equation in three
dimensions with dissipative boundary conditions such that the continuous
problem is well posed. In this paper, we construct meshes for popular (low
order) Galerkin finite element discretizations such that the discrete system
matrix becomes singular and the discrete problem is not well posed. This
implies that a condition \textquotedblleft the finite element space has to be
sufficiently rich\textquotedblright\ in the form of a \textit{resolution
condition} -- typically imposed for discrete well-posedness -- is not an
artifact from the proof by a compact perturbation argument but necessary for
discrete stability of the Galerkin discretization.

\end{abstract}



\section{Introduction}

The finite element discretization of partial differential equations (PDEs)
modeling time-harmonic wave propagation is a challenging problem in numerical
analysis. It is especially difficult to develop optimal stability and
convergence results because the natural variational formulations of such PDE
models are not coercive. In particular, this lack of coercivity leads to
linear systems with indefinite matrices, which must be numerically solved. It
is well-known that if the discretization space is sufficiently
\textquotedblleft rich\textquotedblright, the discretized problem has a unique
solution if the original PDE model is well-posed. However, what happens for
finite element spaces that are not large enough is less clear. In this work,
we construct meshes such the system matrices resulting from finite element
discretizations {are} singular for certain wavenumbers, even if the original
PDE model is well-posed.

Given a polytopic domain $\Omega\subset\mathbb{R}^{d}$ ($d=2$ or $3$) with
diameter $\ell$, {a} right-hand side in the dual energy space $f\in
\mathcal{H}(\Omega)^{\star}$, and {a wavenumber} $k>0$, we consider
time-harmonic wave propagation problems of the form: Find $u\in\mathcal{H}%
(\Omega)$ {such that}
\begin{equation}
-k^{2}(u,v)_{\Omega}-ik(\gamma_{\mathcal{D}}u,\gamma_{\mathcal{D}}%
v)_{\partial\Omega}+(\mathcal{D}u,\mathcal{D}v)_{\Omega}=\langle
f,v\rangle\qquad\forall v\in\mathcal{H}(\Omega). \label{eq_helmholtz_intro}%
\end{equation}
In~\eqref{eq_helmholtz_intro}, $\mathcal{D}$ is a differential operator from
the de Rham complex (i.e., the gradient, the curl or the divergence),
$\gamma_{\mathcal{D}}$ is the naturally associated trace operator, and
$\mathcal{H}(\Omega)$ is the corresponding Sobolev space. These notation are
rigorously introduced in Section~\ref{SecSetting} below.

The dissipation introduced on the boundary ensures
that~\eqref{eq_helmholtz_intro} is well-posed for all $k>0$. It is furthermore
possible to control the norm of the solution $u$ by the norm of the right-hand
side $f$ explicitly in $k$, see
e.g.~\cite{chaumontfrelet_moiola_spence_2023a,esterhazy_melenk_2011a,hiptmair_moiola_perugia_2010a,MelenkDiss,verfurth_2019a}%
. We focus on so-called \textquotedblleft Robin\textquotedblright\ boundary
conditions for simplicity in this introduction, but the results of this work
are valid for generic boundary conditions, including Dirichlet-to-Neumann
operator{s} arising from the Sommerfeld radiation
condition~\cite{chandlerwilde_monk_2008a,spence_2014a}.

We are interested in the conforming finite element method (FEM) applied
to~\eqref{eq_helmholtz_intro} based on a finite-dimensional subspace
$\mathcal{H}_{h}\subset\mathcal{H}(\Omega)$. In this setting, the goal is to
find a discrete approximation $u_{h}\in\mathcal{H}_{h}$ such that
\begin{equation}
-k^{2}(u_{h},v_{h})_{\Omega}-ik(\gamma_{\mathcal{D}}u_{h},\gamma_{\mathcal{D}%
}v_{h})_{\partial\Omega}+(\mathcal{D}u_{h},\mathcal{D}v_{h})_{\Omega}=\langle
f,v_{h}\rangle\qquad\forall v_{h}\in\mathcal{H}_{h}.
\label{eq_helmholtz_fem_intro}%
\end{equation}
This is a finite-dimensional linear system and, given a basis of
$\mathcal{H}_{h}$, the left-hand side of~\eqref{eq_helmholtz_fem_intro} can be
represented by the system matrix $A$.

Finite element spaces are based on a simplicial mesh $\mathcal{T}_{h}$ of
$\Omega$ and $\mathcal{H}_{h}$ collects all the elements of $\mathcal{H}%
(\Omega)$ whose restriction to each simplex belongs to a suitable polynomial
subspace $\widehat{P}$. This is detailed in Section~\ref{SecFEM}. The key
parameters controlling the design of the finite element space are the diameter
$h$ of the largest simplex in $\mathcal{T}_{h}$ and the maximal polynomial
degree $p$ of the elements of $\widehat{P}$. The quantity $kh/p$ is then
inversely proportional to the number of degrees of freedom per wavelength.

Following~\cite{bernkopf_chaumontfrelet_melenk_2025a,chaumontfrelet_2019a,chaumontfrelet_galkowski_spence_2024a,du_wu_2015_preasymptotic,lu_wu_2025_preasymptotic_maxwell,MelenkDiss,mm_stas_helm2,melenk_sauter_2024a,Schatz74}%
, there exist constants $\beta,c>0$ independent of $h$ and $k$ such that if
\begin{equation}
\frac{kh}{p}\leq c(k\ell)^{\beta} \label{eq_resolution_condition}%
\end{equation}
then the matrix $A$ is known to be regular, and the approximation $u_{h}$ to
$u$ is uniformly accurate for all $k$. The particular value of $\beta$ depends
on $p$ and the geometry of $\Omega$, as we recap in Section~\ref{SecFEM}. As a
general rule of thumb, the values of $\beta$ decreases with $p$, which makes
high-order methods more suitable. It is numerically observed that the
resolution condition in~\eqref{eq_resolution_condition} is necessary to obtain
accurate approximations. In fact, under appropriate assumptions, the optimal
value of $\beta$ is $1/(2p)$.
%This is known as the pollution effect.

In this work, we are interested in the properties of the matrix $A$ when no
resolution conditions as in~\eqref{eq_resolution_condition} are imposed.
Although such discretization does not lead to accurate approximations, it is
still of interest in the context of
preconditioning~\cite{erlangga_oosterlee_vuik_2006a} and
adaptivity~\cite{bespalov_haberl_praetorius_2017a}. Indeed, high accuracy is
not expected for the coarsest level of a multigrid solver nor for the initial
mesh of an adaptive algorithm performing iterative refinements. Rather, one
desires that the matrix $A$ possesses basic algebraic and analytical
properties and, in particular, that it is non-singular.

In \cite[Theorem 3.1]{bernkopf2021solvability} a positive result in one space
dimension has been proved: Any $hp$ finite element Galerkin discretization of
the one-dimensional Helmholtz equation with Robin boundary condition has a
unique solution for all wavenumbers $k\in\mathbb{R}\backslash\left\{
0\right\}  $ and all one-dimensionsal meshes.

In this paper we investigate the well-posedness of the Galerkin discretization
for time-harmonic high-frequency scattering problems in two and three
dimensions. The results we provide are of a negative nature. Specifically, we
construct a mesh $\widehat{\mathcal{T}}$ for which the system matrix $A$ is
singular for particular values of $k$. Our work is inspired
by~\cite{bernkopf2021solvability} where one mesh was presented such that the
lowest-order Lagrange finite element in two space dimensions becomes singular.
Here, we expand on~\cite{bernkopf2021solvability} and construct further meshes
such that the Galerkin discretization becomes singular, more precisely we show
that for the Galerkin method on these meshes for (a) two- and
three-dimensional settings, (b) Lagrange, Raviart--Thomas and N\'{e}d\'{e}lec
finite elements, and (c) larger polynomial degrees the arising linear system
can become singular.

Our examples do not only concern the particular mesh $\widehat{\mathcal{T}}$,
but in fact, any mesh $\mathcal{T}_{h}$ containing (a possibly scaled version
of) $\widehat{\mathcal{T}}$ as a submesh is affected. To be more specific, we
in fact construct finite element functions $u_{h}\in\mathcal{H}_{h}%
\setminus\{0\}$ with local support such that $(\mathcal{D}u_{h},\mathcal{D}%
v_{h})_{\Omega}=k^{2}(u_{h},v_{h})_{\Omega}$ for all $v_{h}\in\mathcal{H}_{h}%
$. This is a manifestation of the fact that the so-called \textquotedblleft
unique continuation principle\textquotedblright\ does not hold at the discrete
level~\cite{cox_maclachlan_steeves_2025a}. This is in contrast to the
continuous level, where such a principle is in fact used to establish the
well-posedness of~\eqref{eq_helmholtz_intro}, see
e.g.~\cite{burq_1998a,wolff_1992_UCP}.

{The numerical discretization of unique continuation problems is also of
independent interest (see e.g.~\cite{burman_delay_ern_2021a}), and we expect
that the present results might be of interest in this context.}

The remainder of this work is organized as follows. In Section
\ref{SecSetting}, three model problems for high-frequency scattering are
introduced on the continuous level and well-posedness results from the
literature are recalled. The Galerkin discretizations for these models by
popular finite element spaces are presented in Section \ref{SecFEM} and their
well-posedness is formulated provided the finite element space satisfies a
resolution condition which is stated in an explicit way.

In Section \ref{SecBern} we introduce the two basic meshes, one for two and
one for three dimensions, which are used to construct the singular
discretizations. They are inspired by the work \cite{bernkopf2021solvability}
and we denote them by \textit{Bernkopf meshes.} We assemble the system
matrices by an implementation\footnote{The source code of the
\textsc{Mathematica }notebook and data files are online available, Link:
https://drive.math.uzh.ch/index.php/s/daFi9g27CZ3eXmC} of our finite element
methods in the symbolic mathematics software \textsc{Mathematica}%
$^{\text{\texttrademark}}$ (see \cite{Mathematica}). The matrix entries then
are quadratic functions of the wavenumber $k$. To analyze the critical
wavenumbers leading to a singular system matrix we have computed its
determinant (a polynomial in $k$) and analyzed its roots. Here, we present the
results by explicitly characterizing all the critical frequencies $k$ for
these meshes and different examples.

In the concluding Section \ref{SecConcl} we summarize our findings and also
explain that general meshes which contain a (scaled and dilated) Bernkopf mesh
as a submesh lead to singular system matrices for scaled versions of the
critical wavenumbers $k$ in the basic Bernkopf meshes.

For the Maxwell equation, discretized by N\'{e}d\'{e}lec elements, the
analysis of the determinant of the system matrix involves a discussion of a
cubic polynomial with parameter-dependent coefficients. We have shifted this
tedious analysis to Appendix \ref{ApTech}.

\section{Setting\label{SecSetting}}

In this section we introduce our model scattering problems and their
discretization. The physical domain $\Omega\subset\mathbb{R}^{d}$,
$d\in\left\{  2,3\right\}  $ is assumed to be bounded {and Lipschitz}. Its
boundary is denoted by $\Gamma:=\partial\Omega$. and the outer normal vector
field by $\mathbf{n}$.

Let $L^{p}\left(  \Omega\right)  $, $p\in\left[  1,\infty\right]  $, and
$H^{1}\left(  \Omega\right)  $ be the standard Lebesgue and Sobolev spaces on
$\Omega$ and we denote by $\mathbf{L}^{p}\left(  \Omega\right)  :=\left(
L^{p}\left(  \Omega\right)  \right)  ^{d}$, $\mathbf{H}^{1}\left(
\Omega\right)  :=\left(  H^{1}\left(  \Omega\right)  \right)  ^{d}$ their
vector valued analogues. We will also need the spaces%
\begin{align*}
\mathbf{H}\left(  \operatorname*{div};\Omega\right)   &  :=\left\{
\mathbf{v}\in\mathbf{L}^{2}\left(  \Omega\right)  \mid\operatorname*{div}%
\mathbf{v}\in L^{2}\left(  \Omega\right)  \right\}  ,\\
\mathbf{H}\left(  \operatorname*{curl};\Omega\right)   &  :=\left\{
\mathbf{v}\in\mathbf{L}^{2}\left(  \Omega\right)  \mid\operatorname*{curl}%
\mathbf{v}\in\mathbf{L}^{2}\left(  \Omega\right)  \right\}
\end{align*}
and the space of tangential fields in $\mathbf{L}^{2}\left(  \Gamma\right)  $:%
\[
\mathbf{L}_{T}^{2}(\Gamma):=\left\{  \mathbf{v}\in\mathbf{L}^{2}(\Gamma
)\mid\mathbf{n}\cdot\mathbf{v}=0\text{ on }\Gamma\right\}  .
\]
For the spaces $H^{1}\left(  \Omega\right)  $, $\mathbf{H}\left(
\operatorname*{div};\Omega\right)  $, $\mathbf{H}\left(  \operatorname*{curl}%
;\Omega\right)  $ we introduce norms where the $L^{2}$ term is weighted by the
wavenumber $k\in\mathbb{R}\backslash\left\{  0\right\}  $%
\[%
\begin{tabular}
[c]{ll}%
$\left\Vert u\right\Vert _{\nabla,k}:=\left(  \left\Vert \nabla u\right\Vert
_{\mathbf{L}^{2}\left(  \Omega\right)  }^{2}+k^{2}\left\Vert u\right\Vert
_{L^{2}\left(  \Omega\right)  }^{2}\right)  ^{1/2}$ & for $H^{1}\left(
\Omega\right)  ,$\\
$\left\Vert \mathbf{u}\right\Vert _{\operatorname*{div},k}:=\left(  \left\Vert
\operatorname*{div}\mathbf{u}\right\Vert _{L^{2}\left(  \Omega\right)  }%
^{2}+k^{2}\left\Vert \mathbf{u}\right\Vert _{\mathbf{L}^{2}\left(
\Omega\right)  }^{2}\right)  ^{1/2}$ & for $\mathbf{H}\left(
\operatorname{div};\Omega\right)  ,$\\
$\left\Vert \mathbf{u}\right\Vert _{\operatorname*{curl},k}:=\left(
\left\Vert \operatorname*{curl}\mathbf{u}\right\Vert _{\mathbf{L}^{2}\left(
\Omega\right)  }^{2}+k^{2}\left\Vert \mathbf{u}\right\Vert _{\mathbf{L}%
^{2}\left(  \Omega\right)  }^{2}\right)  ^{1/2}$ & for $\mathbf{H}\left(
\operatorname*{curl};\Omega\right)  .$%
\end{tabular}
\ \ \
\]

\begin{problem}
[Helmholtz problem]\label{HelmProb}The weak formulation of the Helmholz
problem with Robin boundary condition is given by%
\begin{equation}
\text{find }u\in H^{1}\left(  \Omega\right)  \text{ s.t.}\quad a_{k}%
^{\operatorname*{H}}\left(  u,v\right)  =F\left(  v\right)  \quad\forall v\in
H^{1}\left(  \Omega\right)  , \label{Helm0}%
\end{equation}
for some given continuous anti-linear form $F\in\left(  H^{1}\left(
\Omega\right)  \right)  ^{\prime}$ and sesquilinear form $a_{k}%
^{\operatorname*{H}}:H^{1}\left(  \Omega\right)  \times H^{1}\left(
\Omega\right)  \rightarrow\mathbb{C}$ defined by%
\[
a_{k}^{\operatorname*{H}}\left(  u,v\right)  :=\left(  \nabla u,\nabla
v\right)  _{\mathbf{L}^{2}\left(  \Omega\right)  }-k^{2}\left(  u,v\right)
_{L^{2}\left(  \Omega\right)  }-\operatorname*{i}k\left(  u,v\right)
_{L^{2}\left(  \Gamma\right)  }.
\]

\end{problem}

Next, we consider the transverse magnetic problem in 2D with impedance
boundary condition. The \emph{energy space} for this problem is chosen to be%
\begin{equation}
\mathbf{V}_{\operatorname*{imp}}:=\left\{  \mathbf{u}\in\mathbf{H}\left(
\operatorname{div};\Omega\right)  \mathbf{\mid u}\cdot\mathbf{n}\in
L^{2}(\Gamma)\right\}
\end{equation}
and furnished {with} the norm%
\[
\left\Vert \mathbf{u}\right\Vert _{\operatorname*{div},k,+}:=\left(
\left\Vert \mathbf{u}\right\Vert _{\operatorname*{div},k}^{2}+\left\vert
k\right\vert \left\Vert \mathbf{u}\cdot\mathbf{n}\right\Vert _{L^{2}(\Gamma
)}^{2}\right)  ^{1/2}.
\]

\begin{problem}
[Transverse magnetic problem in 2D]\label{divdivProb}The weak form of the
transverse magnetic problem in two dimensions is given by%
\begin{equation}
\text{find }\mathbf{u}\in\mathbf{V}_{\operatorname*{imp}}\ \text{s.t.}\quad
a_{k}^{\operatorname*{M2D}}\left(  \mathbf{u},\mathbf{v}\right)  =F\left(
\mathbf{v}\right)  \quad\forall\mathbf{v}\in\mathbf{V}_{\operatorname*{imp}},
\label{TMP}%
\end{equation}
for some given continuous anti-linear form $F\in\mathbf{V}%
_{\operatorname*{imp}}^{\prime}$ and sesquilinear form $a_{k}%
^{\operatorname*{M2D}}:\mathbf{V}_{\operatorname*{imp}}\times\mathbf{V}%
_{\operatorname*{imp}}\rightarrow\mathbb{C}$ given by%
\[
a_{k}^{\operatorname*{M2D}}\left(  \mathbf{u},\mathbf{v}\right)  :=\left(
\operatorname*{div}\mathbf{u},\operatorname*{div}\mathbf{v}\right)
_{L^{2}\left(  \Omega\right)  }-k^{2}\left(  \mathbf{u},\mathbf{v}\right)
_{\mathbf{L}^{2}\left(  \Omega\right)  }-\operatorname*{i}k\left(
\mathbf{u}\cdot\mathbf{n},\mathbf{v}\cdot\mathbf{n}\right)  _{\mathbf{L}%
^{2}\left(  \Gamma\right)  }.
\]

\end{problem}

Finally, we consider the three-dimensional Maxwell equations with impedance
boundary condition. The energy space is chosen to be%
\begin{equation}
\mathbf{X}_{\operatorname*{imp}}:=\left\{  \mathbf{u}\in\mathbf{H}\left(
\operatorname*{curl};\Omega\right)  \mathbf{\mid u}_{T}\mathbf{\in L}_{T}%
^{2}(\Gamma)\right\}  , \label{eq:Ximp}%
\end{equation}
where $\mathbf{u}_{T}:=\mathbf{n}\times\left(  \left.  \mathbf{u}\right\vert
_{\Gamma}\times\mathbf{n}\right)  $; its norm is given by%
\[
\left\Vert \mathbf{u}\right\Vert _{\operatorname*{curl},k,+}:=\left(
\left\Vert \mathbf{u}\right\Vert _{\operatorname*{curl},k}^{2}+\left\vert
k\right\vert \left\Vert \mathbf{u}_{T}\right\Vert _{\mathbf{L}^{2}(\Gamma
)}^{2}\right)  ^{1/2}.
\]

\begin{problem}
[Maxwell's equation in 3D]\label{curlcurlProb}For some given continuous
anti-linear form $F\in\mathbf{X}_{\operatorname*{imp}}^{\prime}$, the weak
form of the Maxwell equation in 3D with impedance boundary condition is given
by%
\begin{equation}
\text{find }\mathbf{u}\in\mathbf{X}_{\operatorname*{imp}}\ \text{s.t.}\quad
a_{k}^{\operatorname*{M3D}}\left(  \mathbf{u},\mathbf{v}\right)  =F\left(
\mathbf{v}\right)  \quad\forall\mathbf{v}\in\mathbf{X}_{\operatorname*{imp}}
\label{Max0}%
\end{equation}
with sesquilinear form $a_{k}^{\operatorname*{M3D}}:\mathbf{X}%
_{\operatorname*{imp}}\times\mathbf{X}_{\operatorname*{imp}}\rightarrow
\mathbb{C}$ defined by%
\[
a_{k}^{\operatorname*{M3D}}\left(  \mathbf{u},\mathbf{v}\right)  =\left(
\operatorname*{curl}\mathbf{u},\operatorname{cul}\mathbf{v}\right)
_{\mathbf{L}^{2}\left(  \Omega\right)  }-k^{2}\left(  \mathbf{u}%
,\mathbf{v}\right)  _{\mathbf{L}^{2}\left(  \Omega\right)  }-\operatorname*{i}%
k\left(  \mathbf{u}_{T},\mathbf{v}_{T}\right)  _{\mathbf{L}^{2}\left(
\Gamma\right)  }.
\]

\end{problem}

These problems have in common that they are well posed for all wavenumbers
$k\in\mathbb{R}\backslash\left\{  0\right\}  $ and we cite here the relevant theorems.

\begin{remark}
For the well-posedness of the Galerkin discretization of these problems an
approximation property (see (\ref{adapproxHelm}), (\ref{adapproxMW2D}),
(\ref{delta_k_def})) will be key which involves the continuous solution
operator applied to functions in a space which is compactly embedded in the
dual energy space. This is why we define in the following the solution
operator for these subspaces.
\end{remark}

\begin{proposition}
[{\cite[Prop.~8.1.3]{MelenkDiss}}]\label{prop:solvability-robin} Let $\Omega$
be a bounded Lipschitz domain. Then, (\ref{Helm0}) is uniquely solvable for
all $F\in\left(  H^{1}\left(  \Omega\right)  \right)  ^{\prime}$ and the
solution depends continuously on the data.
\end{proposition}

Proposition \ref{prop:solvability-robin} implies that the solution operator
$T_{k}^{\operatorname*{H}}:L^{2}\left(  \Omega\right)  \rightarrow
H^{1}\left(  \Omega\right)  $ is well defined {for $f\in L^{2}$} by%
\begin{equation}
a_{k}^{\operatorname*{H}}\left(  T_{k}^{\operatorname*{H}}f,v\right)  =\left(
f,v\right)  _{L^{2}\left(  \Omega\right)  }\quad\forall v\in H^{1}\left(
\Omega\right)  . \label{solopHelm}%
\end{equation}

\begin{proposition}
[{\cite[Thm. 3.1]{chaumontfrelet_2019a}}]\label{Lem:apriori:TME} Let $\Omega$
be a bounded Lipschitz domain. Then, (\ref{TMP}) is uniquely solvable for all
$F\in\mathbf{V}_{\operatorname*{imp}}^{\prime}$ and the solution depends
continuously on the data\footnote{Note that the general assumption that
$\Omega$ is a convex polygonal domain in \cite{chaumontfrelet_2019a} is not
needed in the proof of Theorem 3.1 therein.}.
\end{proposition}

If the assumptions in Proposition \ref{Lem:apriori:TME} are satisfied, the
solution operator $T_{k}^{\operatorname*{M2D}}:\mathbf{L}^{2}\left(
\Omega\right)  \rightarrow\mathbf{H}\left(  \operatorname*{div};\Omega\right)
$ is well defined by%
\[
a_{k}^{\operatorname*{M2D}}\left(  T_{k}^{\operatorname*{M2D}}\mathbf{f}%
,\mathbf{v}\right)  =\left(  \mathbf{f},\mathbf{v}\right)  _{\mathbf{L}%
^{2}\left(  \Omega\right)  }\quad\forall\mathbf{v}\in\mathbf{H}^{1}\left(
\operatorname*{div};\Omega\right)
\]
{for all $\mathbf{f}\in\mathbf{L}^{2}$}.

Well-posedness of the Maxwell problem with impedance condition is proved in
\cite[Thm.~{4.17}]{Monkbook}.

\begin{proposition}
\label{lemma:apriori-homogeneous-rhs}Let $\Omega\subset\mathbb{R}^{3}$ be a
bounded Lipschitz domain with simply connected and sufficiently smooth
boundary. Then, (\ref{Max0}) is uniquely solvable for all $F\in\mathbf{X}%
_{\operatorname*{imp}}^{\prime}$ and the solution depends continuously on the data.
\end{proposition}

We introduce the space%
\[
\mathbf{V}_{k,0}:=\left\{  \mathbf{v}\in\mathbf{X}_{\operatorname*{imp}}\mid
k^{2}\left(  \mathbf{v},\nabla\varphi\right)  _{\mathbf{L}^{2}\left(
\Omega\right)  }+\operatorname*{i}k\left(  \mathbf{v}_{T},\left(
\nabla\varphi\right)  _{T}\right)  _{\mathbf{L}^{2}\left(  \Gamma\right)
}=0\quad\forall\varphi\in H_{\operatorname*{imp}}^{1}\left(  \Omega\right)
\right\}  ,
\]
where%
\[
H_{\operatorname*{imp}}^{1}(\Omega):=\left\{  \varphi\in H^{1}(\Omega
)\mid\left.  \varphi\right\vert _{\Gamma}\in H^{1}(\Gamma)\right\}  .
\]

The assumptions of Proposition \ref{lemma:apriori-homogeneous-rhs} imply that
the solution operator $T_{k}^{\operatorname*{M3D}}:\mathbf{V}_{k,0}%
\rightarrow\mathbf{X}_{\operatorname*{imp}}$ is well defined by%
\[
a_{k}^{\operatorname*{M3D}}\left(  T_{k}^{\operatorname*{M3D}}\mathbf{f}%
,\mathbf{v}\right)  =\left(  \mathbf{f},\mathbf{v}\right)  _{\mathbf{L}%
^{2}\left(  \Omega\right)  }\quad\forall\mathbf{v}\in\mathbf{X}%
_{\operatorname*{imp}}%
\]
{for all $\mathbf{f}\in\mathbf{V}_{k,0}$.}

\section{The finite element spaces\label{SecFEM}}

A Galerkin method is employed to discretize these equations which will be
based on common finite element spaces on conforming (no hanging nodes)
simplicial meshes $\mathcal{T}$ for $\Omega$ that satisfy

\begin{enumerate}
\item The (closed) elements $K\in{\mathcal{T}}$ cover $\Omega$, i.e.,
$\overline{\Omega}=\cup_{K\in{\mathcal{T}}}K$.

\item Associated with each element $K$ is the \emph{element map}, a $C^{1}%
$-diffeomorphism $\phi_{K}:\widehat{K}\rightarrow K$. The set $\widehat{K}$ is
the \emph{reference simplex}
\[
\widehat{K}:=\left\{  \mathbf{x}\in\mathbb{R}_{\geq0}^{d}\mid\sum_{i=1}%
^{d}x_{i}\leq1\right\}  .
\]

\item Denoting $h_{K}=\operatorname*{diam}K$, there holds, with some
\emph{shape-regularity constant $\gamma_{\mathcal{T}}$},
\begin{equation}
h_{K}^{-1}\left\Vert \phi_{K}^{\prime}\right\Vert _{\mathbf{L}^{\infty}\left(
\widehat{K}\right)  }+h_{K}\left\Vert (\phi_{K}^{\prime})^{-1}\right\Vert
_{\mathbf{L}^{\infty}\left(  \widehat{K}\right)  }\leq\gamma_{\mathcal{T}}.
\label{defhkloc}%
\end{equation}

\item The intersection of two elements is either empty, a vertex, an edge, a
face (for $d=3$), or they coincide (here, vertices, edges, and faces are the
images of the corresponding entities on the reference tetrahedron
$\widehat{K}$).

\item The parametrization of common inner edges (2D) and faces (3D) is
compatible: their pullbacks by $\phi_{K}^{-1}$ and $\phi_{K^{\prime}}^{-1}$
for two adjacent elements $K$, $K^{\prime}$ are affine equivalent.
\end{enumerate}

\begin{remark}
The construction of simplicial finite element meshes that satisfy the above
assumption can {be} based on patchwise structured meshes as described, for
example, in \cite[Ex.~{5.1}]{MelenkSauterMathComp} or \cite[Sec.~{3.3.2}%
]{MelenkHabil}.
\end{remark}

Next, we recall the definition of conforming finite element spaces for the
discretization of Problems \ref{HelmProb} - \ref{curlcurlProb}. For a simplex
$K\in\mathcal{T}$, let $\mathbb{P}_{p}\left(  K\right)  $ denote the space of
$d-$variate polynomials of total degree $\leq p$.

\begin{definition}
[Conforming \textit{hp} finite element space]Let $\mathcal{T}$ be a conforming
finite element mesh for the domain $\Omega$. Then the \emph{conforming }%
$hp$\emph{ finite element space} $S_{p}\left(  \mathcal{T}\right)  \subset
H^{1}\left(  \Omega\right)  $ is given by%
\[
S_{p}\left(  \mathcal{T}\right)  :=\left\{  u\in H^{1}\left(  \Omega\right)
\mid\forall K\in\mathcal{T}\quad\left.  u\right\vert _{K}\circ\phi_{K}%
\in\mathbb{P}_{p}\left(  K\right)  \right\}  .
\]

\end{definition}

We employ this space for the discretization of Problem \ref{HelmProb}:%
\begin{equation}
\text{find }u_{S}\in S_{p}\left(  \mathcal{T}\right)  \ \text{s.t.}\quad
a_{k}^{\operatorname*{H}}\left(  u_{S},v\right)  =F\left(  v\right)
\quad\forall v\in S_{p}\left(  \mathcal{T}\right)  . \label{DiscHelm}%
\end{equation}

\begin{definition}
[Raviart-Thomas elements]On the reference element the \emph{Raviart-Thomas
element} of order $p$ is given by%
\[
\mathbf{RT}_{p}(\widehat{K}):=\left\{  \mathbf{p}\left(  \mathbf{x}\right)
+\mathbf{x}{q}\left(  \mathbf{x}\right)  \,|\,\mathbf{p}\in\left(
\mathbb{P}_{p}\left(  \widehat{K}\right)  \right)  ^{3},{q}\in\mathbb{P}%
_{p}\left(  \widehat{K}\right)  \right\}  .
\]
For a conforming simplicial finite element mesh $\mathcal{T}$ of $\Omega$, the
\emph{Raviart-Thomas finite element} of order $p$ is the space:%
\[
\mathbf{RT}_{p}\left(  {\mathcal{T}}\right)  :=\left\{  \mathbf{u}%
\in\mathbf{H}\left(  \operatorname{div};\Omega\right)  \mid\left(
{\operatorname*{det}\phi_{K}^{\prime}}\right)  \left(  \phi_{K}^{\prime
}\right)  ^{-1}\mathbf{u}|_{K}\circ\phi_{K}\in\mathbf{RT}_{p}\left(
\widehat{K}\right)  \right\}  .
\]

\end{definition}

The Raviart-Thomas finite element is employed for the discretization of
Problem \ref{divdivProb}:%
\begin{equation}
\text{find }\mathbf{u}_{\operatorname*{RT}}\in\mathbf{RT}_{p}\left(
{\mathcal{T}}\right)  \ \text{s.t.}\quad a_{k}^{\operatorname*{M2D}}\left(
\mathbf{u}_{\operatorname*{RT}},\mathbf{v}\right)  =F\left(  \mathbf{v}%
\right)  \quad\forall\mathbf{v}\in\mathbf{RT}_{p}\left(  {\mathcal{T}}\right)
. \label{DiscM2D}%
\end{equation}

\begin{definition}
[N\'{e}d\'{e}lec elements]On the reference element $\widehat{K}$ the
\emph{N\'{e}d\'{e}lec (type I) element} of order $p$ is given by%
\[
\mathbf{N}_{p}^{\operatorname*{I}}\left(  \widehat{K}\right)  :=\left\{
\mathbf{p}\left(  \mathbf{x}\right)  +s\mathbf{x}\times\mathbf{q}\left(
\mathbf{x}\right)  \mid\mathbf{p},\mathbf{q}\in\left(  \mathbb{P}_{p}\left(
\widehat{K}\right)  \right)  ^{3}\right\}  .
\]
For a \emph{conforming simplicial finite element mesh} $\mathcal{T}$ of
$\Omega$, the \emph{N\'{e}d\'{e}lec (type I) finite element} of order $p$ is
the space:%
\[
\mathbf{N}_{p}^{\operatorname*{I}}\left(  {\mathcal{T}}\right)  :=\left\{
\mathbf{u}\in\mathbf{H}\left(  \operatorname*{curl};\Omega\right)  \mid\left(
\phi_{K}^{\prime}\right)  ^{T}\mathbf{u}|_{K}\circ\phi_{K}\in\mathbf{N}%
_{p}^{\operatorname*{I}}\left(  \widehat{K}\right)  \right\}  .
\]

\end{definition}

We consider discretizations of Problem \ref{curlcurlProb} by a Galerkin method
based on the N\'{e}d\'{e}lec space $\mathbf{N}_{p}^{\operatorname*{I}}\left(
{\mathcal{T}}\right)  :$%
\begin{equation}
\text{find }\mathbf{u}_{\operatorname*{N}}\in\mathbf{N}_{p}^{\operatorname*{I}%
}\left(  {\mathcal{T}}\right)  \ \text{s.t.}\quad a_{k}^{\operatorname*{M3D}%
}\left(  \mathbf{u}_{\operatorname*{N}},\mathbf{v}\right)  =F\left(
\mathbf{v}\right)  \quad\forall\mathbf{v}\in\mathbf{N}_{p}^{\operatorname*{I}%
}\left(  {\mathcal{T}}\right)  . \label{NdecurlcurlGal}%
\end{equation}

Since the original variational problems (Prob. \ref{HelmProb} --
\ref{curlcurlProb}) are not coercive the well-posedness of the Galerkin
discretization is not simply inherited. The common theory of well-posedness
for the discretizations is based on a compact perturbation argument which is
known as the \textquotedblleft Schatz argument\textquotedblright\ (see
\cite{Schatz74}); for Maxwell's equation the argument is more subtle since
$\mathbf{H}\left(  \operatorname*{curl},\Omega\right)  $ is not compactly
embedded in $\mathbf{L}^{2}\left(  \Omega\right)  $. In any case the
well-posedness theorem for conforming Galerkin discretizations of the original
variational problem is typically of the form: if the finite element space is
\textquotedblleft rich enough\textquotedblright, then the discrete equations
are well-posed. We cite here the relevant theorems but first introduce some notation.

\begin{proposition}
[{\cite[Thm. 3.2]{mm_stas_helm2}}]Let the simplicial finite element mesh
$\mathcal{T}$ and the polynomial degree be chosen such that the adjoint
approximation property (cf. (\ref{solopHelm}))%
\begin{equation}
\eta\left(  S_{p}\left(  \mathcal{T}\right)  \right)  :=\sup_{f\in
L^{2}\left(  \Omega\right)  \backslash\left\{  0\right\}  }\inf_{v\in
S_{p}\left(  \mathcal{T}\right)  }\frac{\left\Vert T_{-k}^{\operatorname*{H}%
}f-v\right\Vert _{\mathcal{H}}}{\left\Vert f\right\Vert _{L^{2}\left(
\Omega\right)  }} \label{adapproxHelm}%
\end{equation}
satisfies%
\[
C_{\operatorname{c}}\eta\left(  S_{p}\left(  \mathcal{T}\right)  \right)
<\frac{1}{2}.
\]
Then the Galerkin discretization (\ref{DiscHelm}) of Problem \ref{HelmProb} is well-posed.
\end{proposition}

For the transverse magnetic problem in 2D with impedance boundary condition we
quote a result in \cite[Thm. 4.1]{chaumontfrelet_2019a} for lowest order
Raviart-Thomas elements which also shows for this equation that if the maximal
mesh width $h$ of the shape regular triangulation $\mathcal{T}$ for
$\mathbf{RT}_{1}\left(  {\mathcal{T}}\right)  $ is small enough, the Galerkin
discretization is well posed.

\begin{proposition}
Let $\Omega\subset\mathbb{R}^{2}$ be a convex bounded polygonal domain. If
\begin{equation}
k^{2}h\text{ is sufficiently small} \label{adapproxMW2D}%
\end{equation}
then the Galerkin discretization (\ref{DiscM2D}) of Problem \ref{divdivProb}
is well posed.
\end{proposition}

For the Maxwell problem the well-posedness for the Galerkin discretization by
N\'{e}d\'{e}lec elements (\ref{NdecurlcurlGal}) is investigated in
\cite[\S 9]{melenk_sauter_2024a}. Here, we formulate that result in a slightly
different way. We introduce the adjoint approximation property $\delta
_{k}:\mathbf{X}_{\operatorname*{imp}}\rightarrow\mathbb{R}$ by $\delta
_{k}\left(  \mathbf{0}\right)  :=0\ $and for $\mathbf{w}\in\mathbf{X}%
_{\operatorname*{imp}}\backslash\left\{  \mathbf{0}\right\}  $ by%
\begin{equation}
\delta_{k}(\mathbf{w}):=2\sup_{\mathbf{v}_{\operatorname*{N}}\in\mathbf{N}%
_{p}^{\operatorname*{I}}\left(  {\mathcal{T}}\right)  \backslash\left\{
\mathbf{0}\right\}  }\frac{\left\vert k^{2}\left(  \mathbf{w},\mathbf{v}%
_{\operatorname*{N}}\right)  _{\mathbf{L}^{2}\left(  \Omega\right)
}+\operatorname*{i}k\left(  \mathbf{w}_{T},\left(  \mathbf{v}%
_{\operatorname*{N}}\right)  _{T}\right)  _{\mathbf{L}^{2}\left(
\Gamma\right)  }\right\vert }{\left\Vert \mathbf{w}\right\Vert
_{\operatorname*{curl},k,+}\left\Vert \mathbf{v}_{\operatorname*{N}%
}\right\Vert _{\operatorname*{curl},k,+}}. \label{delta_k_def}%
\end{equation}

\begin{proposition}
Let the assumptions in Prop. \ref{lemma:apriori-homogeneous-rhs} be satisfied
and let $\mathbf{u}\in\mathbf{X}_{\operatorname*{imp}}$ be the exact solution.
Assume that\ for any given $\alpha\in\left]  0,1\right[  $, the conforming
simplicial finite element mesh $\mathcal{T}$ and the polynomial degree $p$ is
chosen such that for any function $\mathbf{w}\in\mathbf{X}%
_{\operatorname*{imp}}$ with%
\[
a_{k}^{\operatorname*{M3D}}\left(  \mathbf{w},\mathbf{v}\right)
=0\quad\forall\mathbf{v}\in\mathbf{N}_{p}^{\operatorname*{I}}\left(
{\mathcal{T}}\right)
\]
it holds $\delta_{k}\left(  \mathbf{w}\right)  \leq\alpha$. Then, the discrete
problem (\ref{NdecurlcurlGal}) has a unique solution $\mathbf{u}%
_{\operatorname*{N}}\in\mathbf{N}_{p}^{\operatorname*{I}}\left(
\mathcal{T}\right)  $ which satisfies the quasi-optimal error estimate%
\[
\left\Vert \mathbf{u}-\mathbf{u}_{\operatorname*{N}}\right\Vert
_{\operatorname*{curl},k,+}\leq\frac{1+\delta_{k}(\mathbf{e}%
_{\operatorname*{N}})}{1-\delta_{k}(\mathbf{e}_{\operatorname*{N}})}%
\inf_{\mathbf{v}_{\operatorname*{N}}\in\mathbf{N}_{p}^{\operatorname*{I}%
}\left(  \mathcal{T}\right)  }\left\Vert \mathbf{u}-\mathbf{v}%
_{\operatorname*{N}}\right\Vert _{\operatorname*{curl},k,+}\quad\text{for
}\mathbf{e}_{\operatorname*{N}}:=\mathbf{u}-\mathbf{u}_{\operatorname*{N}}.
\]

\end{proposition}

The proof is a simple reformulation of \cite[Thm. 9.7, Prop. 9.1]%
{melenk_sauter_2024a} which is based on \cite{MelenkLoehndorf}.

\begin{remark}
In \cite[Lem. 9.6]{melenk_sauter_2024a} it is explained how the polynomial
degree $p$ and the mesh have to be chosen such that $\delta_{k}\left(
\mathbf{w}\right)  \leq\alpha$ holds. An essential ingredient of this theory
is to prove that the adjoint approximation property%
\[
\eta\left(  \mathbf{N}_{p}^{\operatorname*{I}}\left(  {\mathcal{T}}\right)
\right)  :=\sup_{\mathbf{v}_{0}\in\mathbf{V}_{-k,0}\backslash\left\{
0\right\}  }\inf_{\mathbf{w}_{\operatorname*{N}}\in\mathbf{N}_{p}%
^{\operatorname*{I}}\left(  \mathcal{T}\right)  }\frac{\left\Vert
T_{-k}^{\operatorname*{M3D}}\mathbf{v}_{0}-\mathbf{w}_{\operatorname*{N}%
}\right\Vert _{\operatorname*{curl},k,+}}{\left\Vert \mathbf{v}_{0}\right\Vert
_{\operatorname*{curl},k,+}}%
\]
can be made small for sufficiently fine meshes $\mathcal{T}$ and polynomial
degree $p$ depending logarithmically on $k$. In this way, the notion
\textquotedblleft$\mathbf{N}_{p}^{\operatorname*{I}}\left(  \mathcal{T}%
\right)  $ is sufficiently rich\textquotedblright\ becomes a qualitative condition.
\end{remark}

In summary, we have illustrated that for some high frequency scattering
problems which are well-posed on the continuous level, the conforming Galerkin
discretization is well-posed if the finite element space is \textquotedblleft
rich enough\textquotedblright\ such that an adjoint approximation property is
sufficiently small.

Hence, it is a natural question whether the conforming Galerkin discretization
can be singular if the resolution condition is not satisfied or whether this
is an artefact of the theory based on the Schatz perturbation argument.
Interestingly, for the Helmholtz equation with Robin boundary condition on a
one-dimensional interval $\Omega=\left(  a,b\right)  $ it has been proven in
\cite{bernkopf2021solvability} that the conforming Galerkin discretization
always admits a unique solution without any restriction to the mesh and/or the
polynomial degree.

The resolution condition is sufficient for discrete well-posedness but leaves
the question open whether there exist meshes such that the Galerkin
discretization becomes singular. This is a motivation for the development of
various modifications of the conforming Galerkin discretization which are
always well-posed; among them is the first order system least squares (FOSLS)
finite element method (see \cite{chen-qiu17}, \cite{Lee_Manteuffel_FOSLS},
\cite{Bernkopf_FOSLS_InProc}), stabilized methods \cite{FW09}, \cite{zhu-wu13}%
, \cite{chen-lu-xu13}, and Discontinuous Petrov Galerkin methods
\cite{Petridis_DPG_Helm}. These methods have in common that the resulting
discretization always have a unique solution. However, the number of degrees
of freedom can be substantially larger than for the conforming Galerkin FEM
and their implementation is more subtle. This was one motivation to
investigate the question:%
\[%
\begin{tabular}
[c]{|l|}\hline
$\text{Can the conforming Galerkin FEM lead to singular system matrices}$\\
$\text{for problems (\ref{DiscHelm}), (\ref{DiscM2D}), (\ref{NdecurlcurlGal}%
)?}$\\\hline
\end{tabular}
\
\]

\section{Bernkopf-type meshes\label{SecBern}}

In \cite{bernkopf2021solvability} an example (Bernkopf mesh) is presented for
the Helmholtz problem \ref{HelmProb} in 2D for lowest order elements such that
the conforming Galerkin discretization has a non-trivial homogeneous solution.
In this section, we investigate whether also for higher polynomial degrees, 3D
problems, and Maxwell's equation such examples exist.

\begin{remark}
We have realized the computations for all examples in the symbolic mathematics
program \textsc{Mathematica}\texttrademark\ (cf. \cite{Mathematica}) and hence
no rounding errors occur in our computation.
\end{remark}

First, we introduce the domain and the meshes and start with the
two-dimensional Bernkopf mesh.

\begin{definition}
[2D Bernkopf mesh](see \cite[\S 3.2]{bernkopf2021solvability}) Let
$\Omega=\left(  -1,1\right)  ^{2}$. The vertices in the triangulation are
$\mathbf{P}_{1}^{\Gamma}=\left(  -1,-1\right)  ^{T}$, $\mathbf{P}_{2}^{\Gamma
}=\left(  1,-1\right)  ^{T}$, $\mathbf{P}_{3}^{\Gamma}=\left(  1,1\right)
^{T}$, $\mathbf{P}_{4}^{\Gamma}=\left(  -1,1\right)  ^{T}$, $\mathbf{P}%
_{1}^{\Omega}=\left(  -\alpha,0\right)  ^{T}$, $\mathbf{P}_{2}^{\Omega
}=\left(  0,-\alpha\right)  ^{T}$, $\mathbf{P}_{3}^{\Omega}=\left(
\alpha,0\right)  ^{T}$, $\mathbf{P}_{4}^{\Omega}=\left(  0,\alpha\right)
^{T}$, $\mathbf{P}_{5}^{\Omega}=\left(  0,0\right)  ^{T}$ for a parameter
$\alpha\in\left]  0,1\right[  $. The triangulation $\mathcal{T}_{2}$ consists
of $12$ triangles $T_{j}$, $j\in\left\{  1,2,\ldots,12\right\}  $ as depicted
in Figure \ref{Fig1}.%
%TCIMACRO{\FRAME{ftbpFU}{2.8392in}{2.8392in}{0pt}{\Qcb{Bernkopf mesh in 2D for
%$\alpha=1/2$.}}{\Qlb{Fig1}}{bernkopf_square.eps}%
%width 2.8392in;  height 2.8392in;  depth 0pt;  original-width 3.5008in;
%original-height 3.5008in;  cropleft "0";  croptop "1";  cropright "1";
%cropbottom "0";  filename 'bernkopf_square.eps';file-properties "XNPEU";}} }%
%BeginExpansion
\begin{figure}[ptb]%
\centering
\includegraphics[
height=2.8392in,
width=2.8392in
]%
{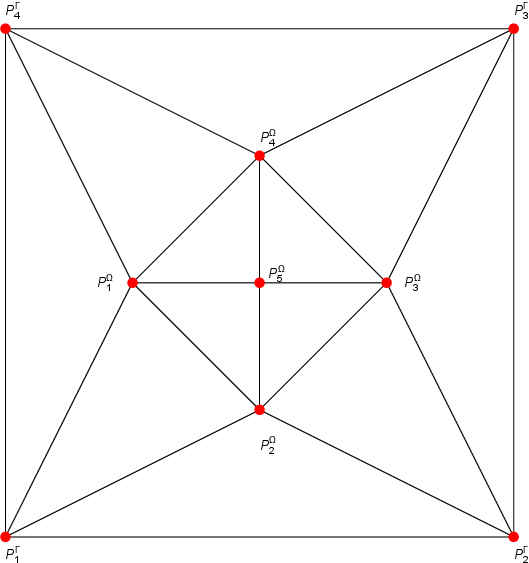}%
\caption{Bernkopf mesh in 2D for $\alpha=1/2$.}%
\label{Fig1}%
\end{figure}
%EndExpansion

\end{definition}

\begin{definition}
[3D Bernkopf mesh]Let $\mathbf{V}_{i}^{\Gamma}=\left(  \mathbf{P}_{i}^{\Gamma
},0\right)  $ for $i\in\left\{  1,2,3,4\right\}  $. Set $\mathbf{V}%
_{5}^{\Gamma}:=\left(  0,0,1\right)  ^{T}$, and $\mathbf{V}_{6}^{\Gamma
}:=\left(  0,0,-1\right)  ^{T}$. The domain $\Omega$ is the convex hull of
$\mathbf{V}_{i}^{\Gamma}$, $i\in\left\{  1,2,\ldots,6\right\}  $ and the inner
vertices are given by $\mathbf{V}_{i}^{\Omega}=\left(  \mathbf{P}_{i}^{\Omega
},0\right)  ^{T}$, $i\in\left\{  1,\ldots,5\right\}  $. The mesh
$\mathcal{T}_{3}$ (see \ref{Fig2}) consists of $24$ tetrahedra of the form
$K_{j}=\operatorname*{conv}\left(  T_{j},\mathbf{V}_{5}^{\Gamma}\right)  $,
$1\leq j\leq12$ and $K_{12+j}=\operatorname*{conv}\left(  T_{j},\mathbf{V}%
_{6}^{\Gamma}\right)  $, $1\leq j\leq12$.%
%TCIMACRO{\FRAME{ftbpFU}{3.8069in}{3.5103in}{0pt}{\Qcb{Bernkopf mesh in 3D for
%the choice $\alpha=1/2$.}}{\Qlb{Fig2}}{bernkopf_diamond.swp.eps}%
%width 3.8069in;  height 3.5103in;  depth 0pt;  original-width 3.7593in;
%original-height 3.4636in;  cropleft "0";  croptop "1";  cropright "1";
%cropbottom "0";  filename 'bernkopf_diamond.swp.eps';file-properties "XNPEU";}%
%} }%
%BeginExpansion
\begin{figure}[ptb]%
\centering
\includegraphics[
height=3.5103in,
width=3.8069in
]%
{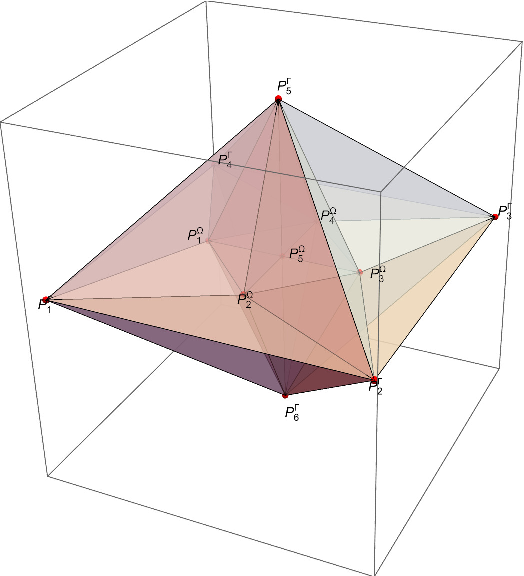}%
\caption{Bernkopf mesh in 3D for the choice $\alpha=1/2$.}%
\label{Fig2}%
\end{figure}
%EndExpansion

\end{definition}

\subsection{Helmholtz Problem \ref{HelmProb}, $d=2$, $p=1,2$}

We explain our conceptual approach for finding critical wavenumbers
$k\in\mathbb{R}\backslash\left\{  0\right\}  $ such that the system matrix for
the Helmholtz problem is singular, while the procedure for Problems
\ref{divdivProb} and \ref{curlcurlProb} is completely analogue. The system
matrix is singular iff there is a non-trivial $u\in S\backslash\left\{
0\right\}  $ to the equation%
\begin{equation}
a^{\operatorname*{H}}\left(  u,v\right)  =0\quad\forall v\in S. \label{homeq}%
\end{equation}
By choosing $v=u$ and considering the imaginary part of the sesquilinear form
for $k\in\mathbb{R}\backslash\left\{  0\right\}  $ leads to $\left\Vert
u\right\Vert _{L^{2}\left(  \Gamma\right)  }=0$ and, in turn to $\left.
u\right\vert _{\Gamma}=0$. For such a function $u$ the boundary part of the
sesquilinear form is zero and we have%
\[
a^{\operatorname*{H}}\left(  u,v\right)  =\left(  \nabla u,\nabla v\right)
_{L^{2}\left(  \Omega\right)  }-k^{2}\left(  u,v\right)  _{L^{2}\left(
\Omega\right)  }.
\]
Let $\left(  b_{y}\right)  _{y\in\Sigma}$ denote the standard local nodal
basis for the space $S$ with the set $\Sigma$ of degrees of freedom. The
degrees of freedom for the inner nodes define the subset $\overset{\circ
}{\Sigma}$. We compute the matrix%
\begin{equation}
A_{\alpha}\left(  k\right)  =\left(  a_{y,z}\left(  k\right)  \right)
_{\substack{y\in\Sigma\\z\in\overset{\circ}{\Sigma}}}\quad\text{with\quad
}a_{y,z}:=\left(  \nabla b_{y},\nabla b_{z}\right)  _{L^{2}\left(
\Omega\right)  }-k^{2}\left(  b_{y},b_{z}\right)  _{L^{2}\left(
\Omega\right)  } \label{defAk}%
\end{equation}
which is a rectangular matrix with more rows than columns. From basi{c} linear
algebra we know that problem (\ref{homeq}) has a non-trivial solution $u\in
S\backslash\left\{  0\right\}  $ iff $\ker A_{\alpha}\left(  k\right)  $ is
non-trivial. This is equivalent to%
\[
\det B_{\alpha}\left(  k\right)  =0\quad\text{for }B_{\alpha}\left(  k\right)
:=\left(  A_{\alpha}\left(  k\right)  \right)  ^{T}A_{\alpha}\left(  k\right)
.
\]
From the definition in (\ref{defAk}) it is clear that $\det B_{\alpha}\left(
k\right)  $ is a polynomial in $\kappa=k^{2}$ so that $Q_{\alpha}\left(
\kappa\right)  :=\det B_{\alpha}\left(  \sqrt{\kappa}\right)  $ is a
polynomial in $\kappa$. To find critical frequencies $k\in\mathbb{R}%
\backslash\left\{  0\right\}  $, we are interested in positive roots of
$Q_{\alpha}\left(  \kappa\right)  $. Our \textsc{Mathematica}\ implementation
computes first the system matrix $A_{\alpha}\left(  k\right)  $, then the
matrix $B_{\alpha}\left(  k\right)  $, and finally its determinant $\det
B_{\alpha}\left(  k\right)  $ as well as the polynomial $Q_{\alpha}$. We
discuss $Q_{\alpha}$ and finally compute positive roots which are the (squares
of the) critical wave numbers $k$. Since $Q_{\alpha}$ is a polynomial the
number of roots is always finite. We derive by symbolic calculus that, for any
$\alpha\in\left]  0,1\right[  $, there exists at least one positive root
$\kappa$ of $Q_{\alpha}$ and hence, at least two spurious frequencies. Since
the $\alpha$-dependent coefficients in $Q_{\alpha}$ are rather complicated we
do not attempt to determine \textit{all} roots of $Q_{\alpha}$ but some of
them.\bigskip

For any $\alpha\in\left]  0,1\right[  $ and $p=1$ our \textsc{Mathematica}
implementation shows that $Q_{\alpha}\left(  \kappa\right)  $ can be
factorized into three polynomials $\eta_{1}\in\mathbb{P}_{1}$, $\eta_{2}%
\in\mathbb{P}_{2}$, $\eta_{3}\in\mathbb{P}_{4}$: $Q_{\alpha}\left(
\kappa\right)  =\frac{\eta_{1}^{2}\eta_{2}^{2}\eta_{3}}{f_{\alpha}}$, where
$f_{\alpha}$ is a positive polynomial for $\alpha\in\left]  0,1\right[  $. The
root of $\eta_{1}$ is given by $\frac{6\left(  2-\alpha\right)  }%
{\alpha\left(  1-\alpha\right)  }$ (positive for $\alpha\in\left]  0,1\right[
$) while the two roots of $\eta_{2}$ have non-zero imaginary part and hence
are not relevant for critical wavenumbers $k\in\mathbb{R}\backslash\left\{
0\right\}  $. We do not analyze the polynomial $\eta_{3}$ furthermore but
simply conclude from the analysis of $\eta_{1}$ that for any $\alpha\in\left]
0,1\right[  $ two critical wavenumbers exist\footnote{Note that this example
for $p=1$ is the one treated in \cite[Lem. 3.2]{bernkopf2021solvability} and
an elementary proof is given without symbolic computation.}%
\[
k\left(  \alpha\right)  ^{2}=\frac{6\left(  2-\alpha\right)  }{\alpha\left(
1-\alpha\right)  }%
\]
and the conforming Galerkin discretization with finite element space
$S_{1}\left(  \mathcal{T}_{2}\right)  $, has a \textit{spurious mode}, i.e.,
there exists an element $u_{p}\in S_{p}\left(  \mathcal{T}_{2}\right)
\backslash\left\{  0\right\}  $ such that%
\[
a_{k\left(  \alpha\right)  }^{\operatorname*{H}}\left(  u_{p},v\right)
=0\quad\forall v\in S_{p}\left(  \mathcal{T}_{2}\right)  .
\]

For any $\alpha\in\left]  0,1\right[  $ and $p=2$, the \textsc{Mathematica}
implementation shows that $Q_{\alpha}\left(  \kappa\right)  $ can be
factorized into five polynomials $\eta_{1}\in\mathbb{P}_{1}$, $\eta_{2}%
\in\mathbb{P}_{4}$, $\eta_{3}\in\mathbb{P}_{6}$, $\eta_{4},\eta_{5}%
\in\mathbb{P}_{10}$: $Q_{\alpha}\left(  \kappa\right)  =f_{\alpha}\eta_{1}%
^{2}\eta_{2}\eta_{3}\eta_{4}\eta_{5}^{2}$, where $f_{\alpha}$ is a positive,
rational function for $\alpha\in\left]  0,1\right[  $. The root of $\eta_{1}$
is given by
\[
k\left(  \alpha\right)  ^{2}=\frac{15\left(  \alpha^{2}+\left(  2-\alpha
\right)  ^{2}\right)  }{\alpha\left(  1-\alpha\right)  \left(  2-\alpha
\right)  }%
\]
(positive for $\alpha\in\left]  0,1\right[  $) while the roots of the other
polynomials are not discussed here. In this way, we have shown that for any
$\alpha\in\left]  0,1\right[  $ two critical wavenumbers exist. The spurious
modes for $p=1,2$ are depicted in Figure \ref{FigHelm2Dp12}.%
%TCIMACRO{\TeXButton{B}{\begin{figure}[tbp] \centering}}%
%BeginExpansion
\begin{figure}[tbp] \centering
%EndExpansion
$%
\begin{array}
[c]{ccc}%
%TCIMACRO{\FRAME{itbpF}{2.4656in}{2.4656in}{0in}{}{}{mode_{p}1.eps}%
%width 2.4656in;  height 2.4656in;  depth 0in;  original-width 28.4558in;
%original-height 28.4558in;  cropleft "0";  croptop "1";  cropright "1";
%cropbottom "0";  filename 'mode_P1.eps';file-properties "XNPEU";}} }%
%BeginExpansion
{\includegraphics[
height=2.4656in,
width=2.4656in
]%
{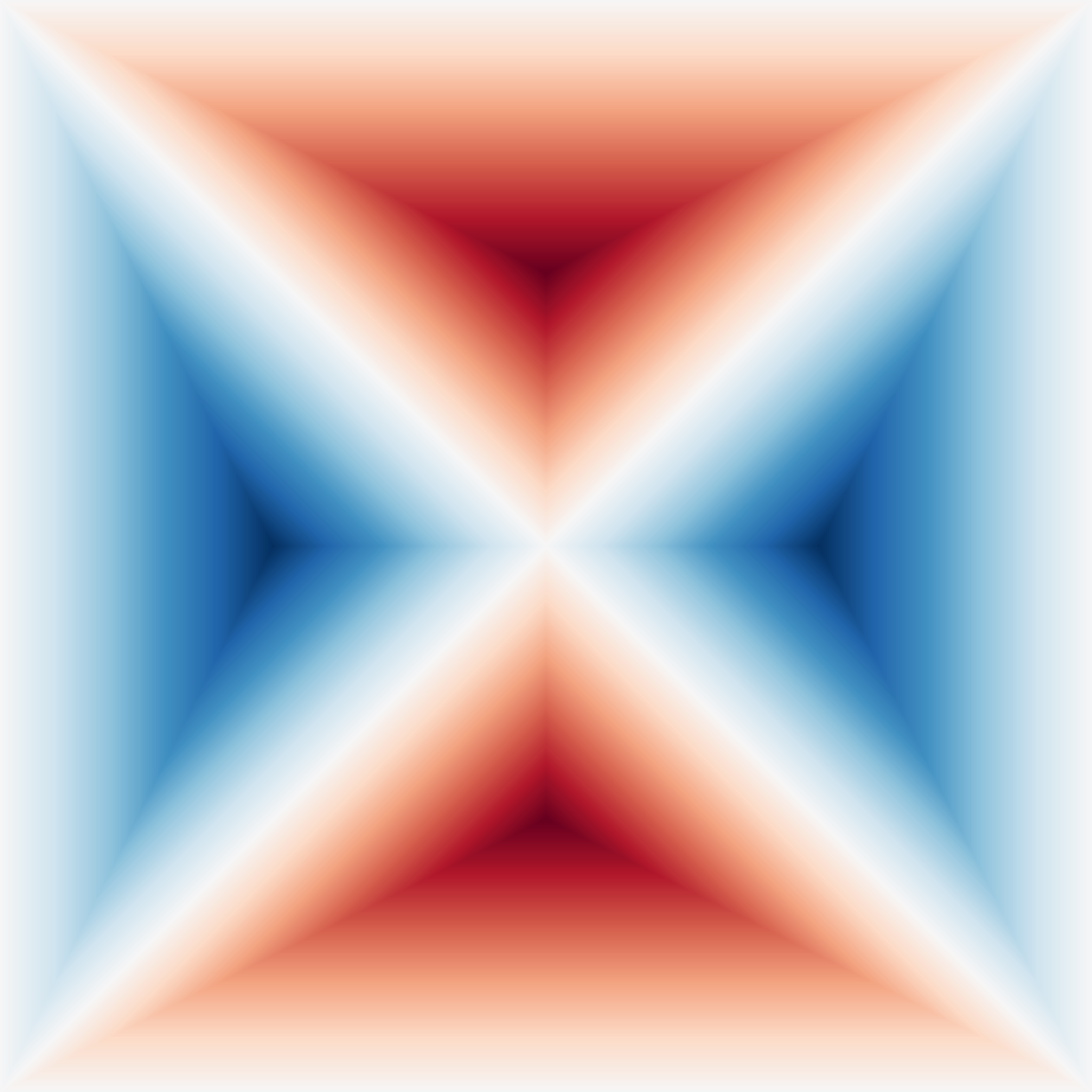}%
}
%EndExpansion
&  &
%TCIMACRO{\FRAME{itbpF}{2.4725in}{2.4725in}{0in}{}{}{mode_{p}2.eps}%
%width 2.4725in;  height 2.4725in;  depth 0in;  original-width 28.4558in;
%original-height 28.4558in;  cropleft "0";  croptop "1";  cropright "1";
%cropbottom "0";  filename 'mode_P2.eps';file-properties "XNPEU";}} }%
%BeginExpansion
{\includegraphics[
height=2.4725in,
width=2.4725in
]%
{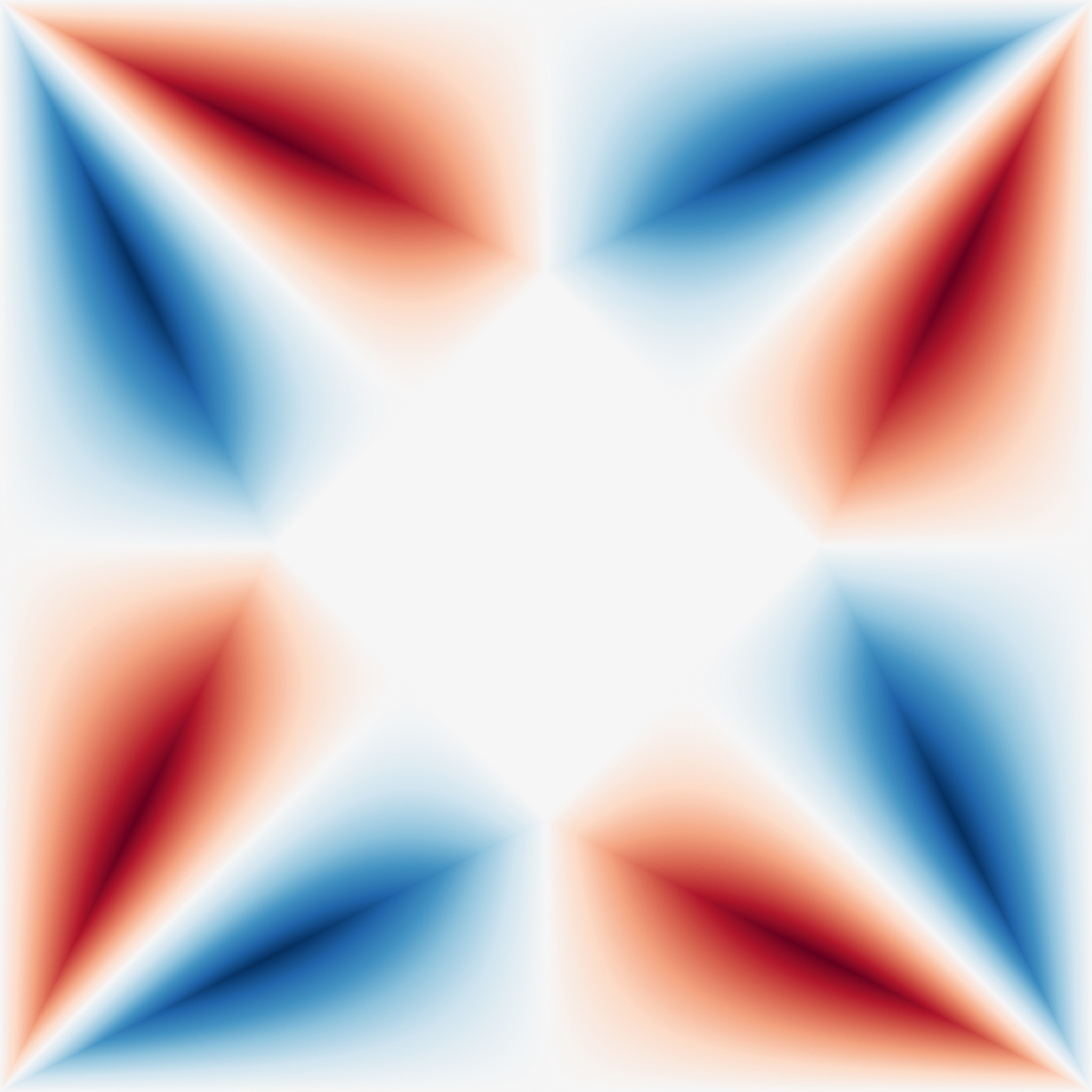}%
}
%EndExpansion
\end{array}
$\caption{Spurious
solutions for the $\mathbb{P}_{1}$ and $\mathbb{P}_{2}$ Galerkin
discretization of the two-dimensional Helmholtz equation for $\alpha=1/2$.
Left: $p=1$ and $k\left(  \alpha\right)  =\pm6$, right: $p=2$ and $k\left(
\alpha\right)  = \pm 10$.\label{FigHelm2Dp12}}%
%TCIMACRO{\TeXButton{E}{\end{figure}}}%
%BeginExpansion
\end{figure}%
%EndExpansion

\subsection{Helmholtz Problem \ref{HelmProb}, $d=3$}

For any $\alpha\in\left]  0,1\right[  $ and $p=1$, symbolic manipulation by
the \textsc{Mathematica} implementation shows that $Q_{\alpha}\left(
\kappa\right)  $ can be factorized into three polynomials $\eta_{1}%
\in\mathbb{P}_{1}$, $\eta_{2}\in\mathbb{P}_{2}$, $\eta_{3}\in\mathbb{P}_{4}:$
$Q_{\alpha}\left(  \kappa\right)  =\frac{\eta_{1}^{2}\eta_{2}^{2}\eta_{3}%
}{f_{\alpha}}$ with a positive polynomial $f_{\alpha}$ for all $\alpha
\in\left]  0,1\right[  $. The root of $\eta_{1}$ is given by $\kappa=\frac
{20}{\alpha\left(  1-\alpha\right)  }$ (positive for $\alpha\in\left]
0,1\right[  $) while the roots of the other polynomials are not discussed
here. In this way, we have shown that for any $\alpha\in\left]  0,1\right[  $
two critical wavenumbers exist%
\[
k\left(  \alpha\right)  ^{2}=\frac{20}{\alpha\left(  1-\alpha\right)  }.
\]

For any $\alpha\in\left]  0,1\right[  $ and $p=2$ symbolic manipulation shows
that $Q_{\alpha}\left(  \kappa\right)  $ can be factorized into eight
polynomials $\eta_{1},\eta_{2}\in\mathbb{P}_{1}$, $\eta_{3}\in\mathbb{P}_{2}$,
$\eta_{4},\eta_{5}\in\mathbb{P}_{4}$, $\eta_{6}\in\mathbb{P}_{8}$, $\eta
_{7}\in\mathbb{P}_{12}$, $\eta_{8}=\mathbb{P}_{14}:Q_{\alpha}\left(
\kappa\right)  =f_{\alpha}\eta_{1}^{2}\eta_{2}^{2}\eta_{3}^{2}\eta_{4}\eta
_{5}\eta_{6}\eta_{7}^{2}\eta_{8}$ with a positive rational function
$f_{\alpha}$ for all $\alpha\in\left]  0,1\right[  $. The root of $\eta_{1}$
is given by $\kappa=\frac{42}{\alpha\left(  1-\alpha\right)  }$ (positive for
$\alpha\in\left]  0,1\right[  $), the root of $\eta_{2}$ by $\frac
{21(\alpha+1)\alpha^{2}+2(2-\alpha)}{(2-\alpha)(1-\alpha)\alpha}$ (positive
for $\alpha\in\left]  0,1\right[  $), the roots of $\eta_{3}$ have non-zero
imaginary part for any $\alpha\in\left]  0,1\right[  $ and not relevant for
critical frequencies $k\in\mathbb{R}\backslash\left\{  0\right\}  $. The
functions $\eta_{4},\ldots,\eta_{8}$ are not discussed. In this way, we have
shown that for any $\alpha\in\left]  0,1\right[  $ four critical wavenumbers
exist%
\[
k\left(  \alpha\right)  ^{2}\in\left\{  \frac{42}{\alpha\left(  1-\alpha
\right)  },\frac{21(\alpha+1)\alpha^{2}+2(2-\alpha)}{(2-\alpha)(1-\alpha
)\alpha}\right\}  .
\]

The spurious solution for $p=1$ is shown in Figure \ref{FigHelm3DP1} while
each of the two cases for $p=2$ are depicted in Figure \ref{FigHelm2DP2}.%
%TCIMACRO{\FRAME{ftbpFU}{3in}{2.0003in}{0pt}{\Qcb{Spurious solution for the
%$\QTR{Bbb}{P}_{1}$ Galerkin discretization of the 3D Helmholtz problem for the
%Bernkopf mesh with $\alpha=1/2$ and critical wavenumber $k_{1}\left(
%1/2\right)  =\pm\sqrt{80}$.}}{\Qlb{FigHelm3DP1}}{p1.eps}%
%maintain-aspect-ratio TRUE;  display "USEDEF";  valid_file "F";  width 3in;
%height 2.0003in;  depth 0pt;  original-width 0pt;  original-height 0pt;
%cropleft "0";  croptop "1";  cropright "1";  cropbottom "0";
%filename 'P1.eps';file-properties "XNPEU";}} }%
%BeginExpansion
\begin{figure}[ptb]%
\centering
\includegraphics[
height=2.0003in,
width=3in
]%
{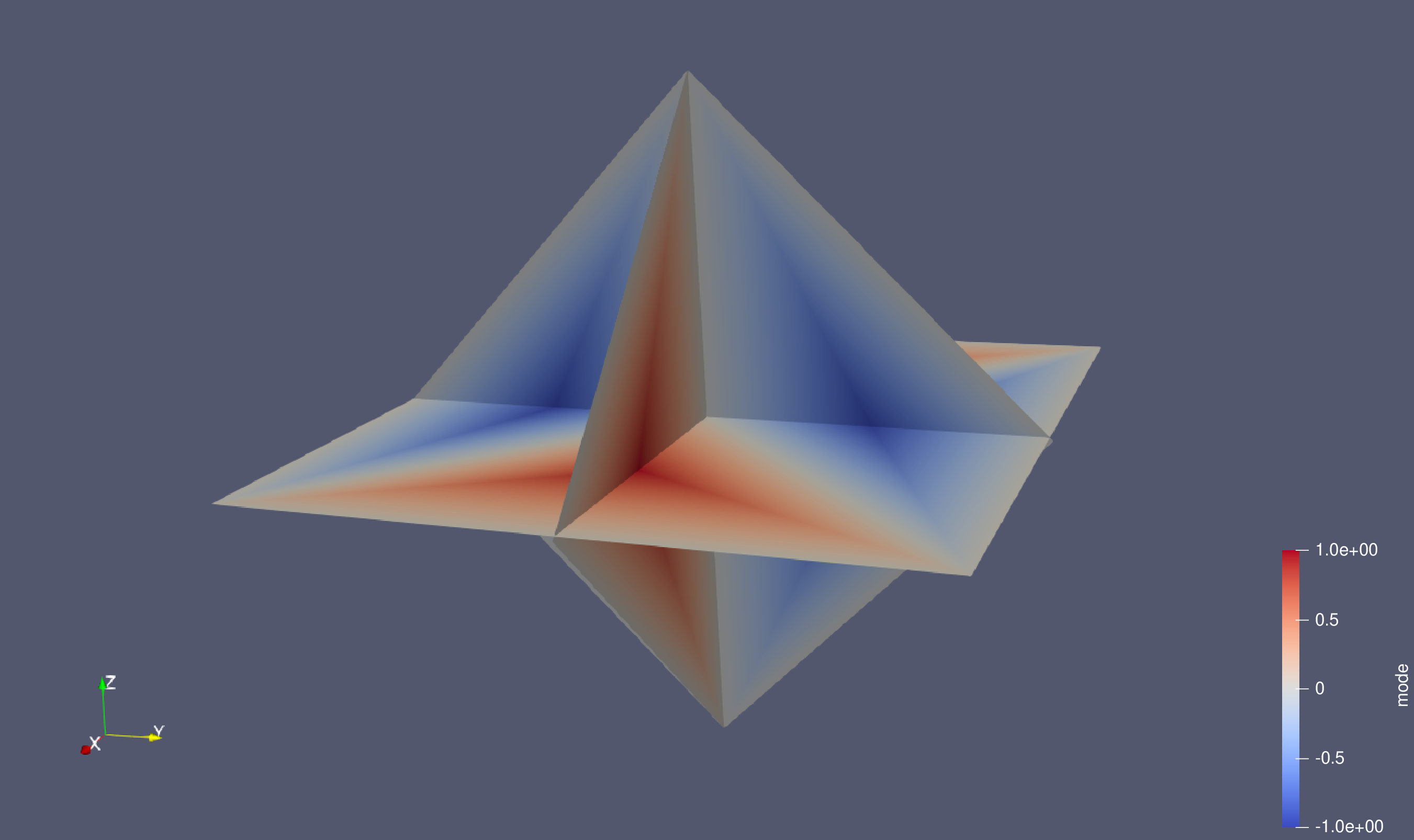}%
\caption{Spurious solution for the $\mathbb{P}_{1}$ Galerkin discretization of
the 3D Helmholtz problem for the Bernkopf mesh with $\alpha=1/2$ and critical
wavenumber $k_{1}\left(  1/2\right)  =\pm\sqrt{80}$.}%
\label{FigHelm3DP1}%
\end{figure}
%EndExpansion
%

%TCIMACRO{\TeXButton{B}{\begin{figure}[tbp] \centering}}%
%BeginExpansion
\begin{figure}[tbp] \centering
%EndExpansion
$%
\begin{array}
[c]{ccc}%
%TCIMACRO{\FRAME{itbpF}{2.5417in}{1.6942in}{0in}{}{}{p2_{1}.eps}%
%width 2.5417in;  height 1.6942in;  depth 0in;  original-width 0pt;
%original-height 0pt;  cropleft "0";  croptop "1";  cropright "1";
%cropbottom "0";  filename 'P2_1.eps';file-properties "XNPEU";}} }%
%BeginExpansion
{\includegraphics[
height=1.6942in,
width=2.5417in
]%
{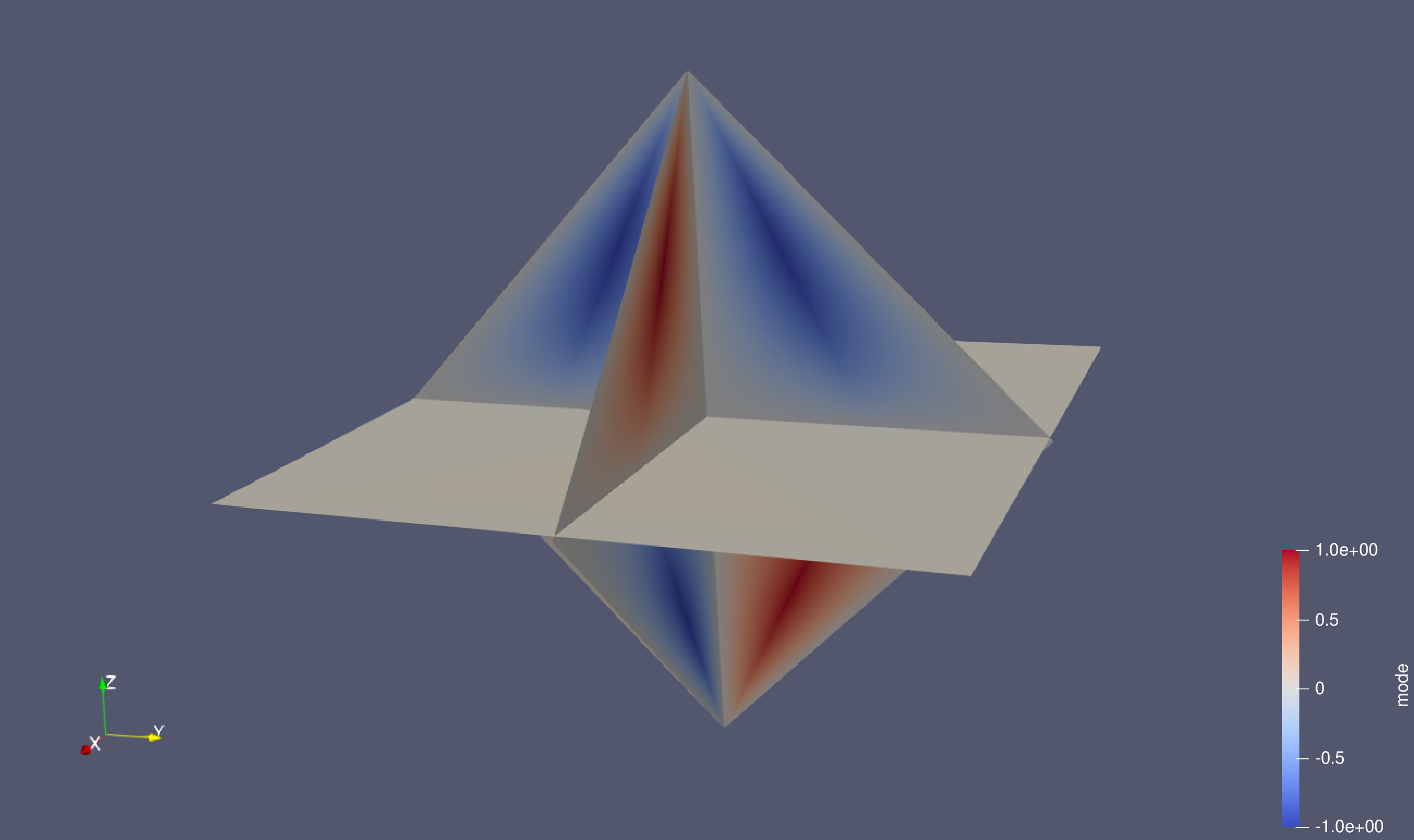}%
}
%EndExpansion
&  &
%TCIMACRO{\FRAME{itbpF}{2.5348in}{1.6898in}{0in}{}{}{p2_{2}.eps}%
%width 2.5348in;  height 1.6898in;  depth 0in;  original-width 0pt;
%original-height 0pt;  cropleft "0";  croptop "1";  cropright "1";
%cropbottom "0";  filename 'P2_2.eps';file-properties "XNPEU";}} }%
%BeginExpansion
{\includegraphics[
height=1.6898in,
width=2.5348in
]%
{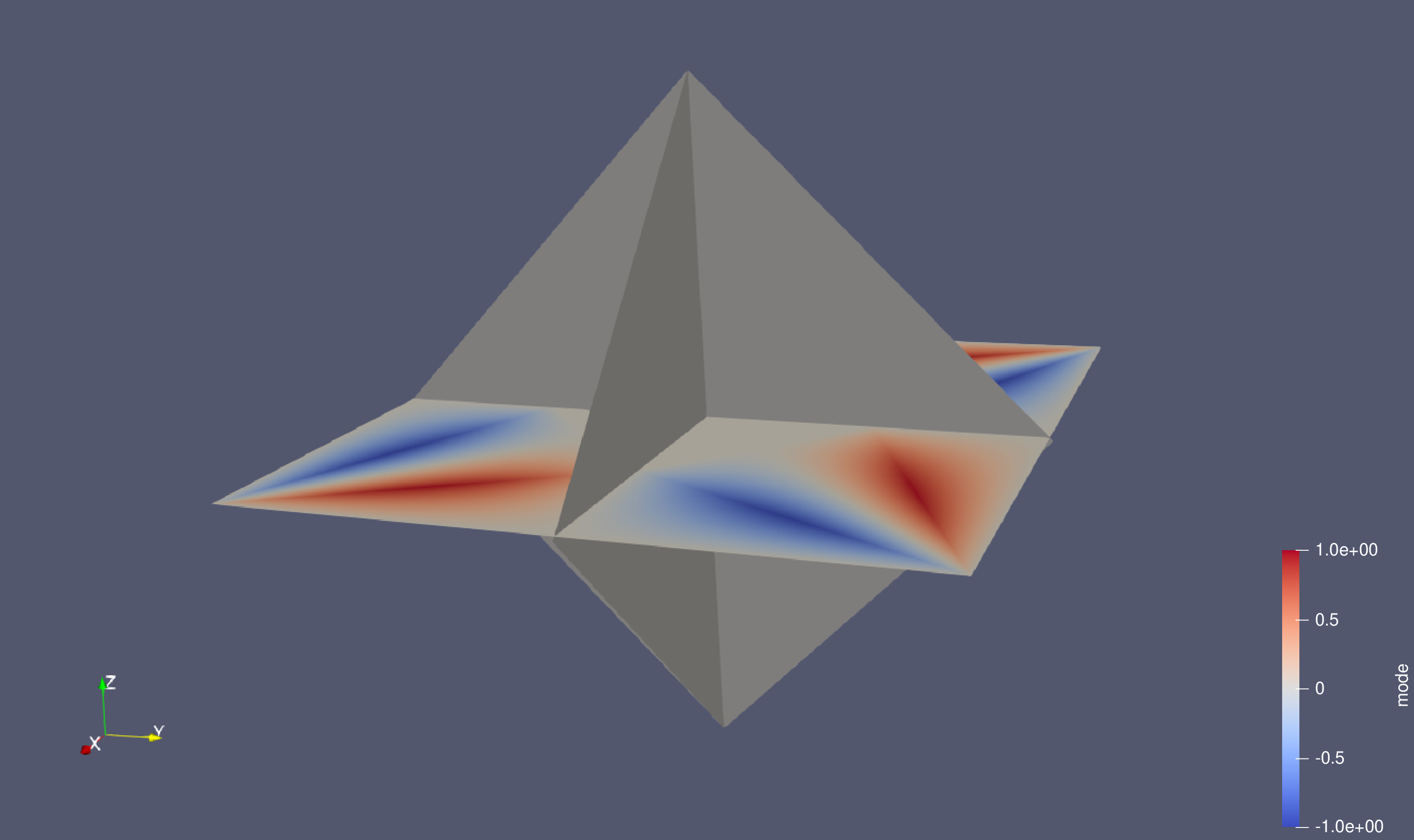}%
}
%EndExpansion
\end{array}
$\caption{Spurious solutions for the $\mathbb{P}_{2}$ Galerkin
discretization for the Helmholtz problem in 3D for $\alpha=1/2$. Left:
$k_{2}\left(  1/2\right)  ^{2}=168$, right: $k_{2}\left(  1/2\right)
^{2}=189$.\label{FigHelm2DP2}}%
%TCIMACRO{\TeXButton{E}{\end{figure}}}%
%BeginExpansion
\end{figure}%
%EndExpansion

\subsection{Maxwell 2D, RT1:}

For any $\alpha\in\left]  0,1\right[  $ and $p=1$, symbolic manipulation by
the \textsc{Mathematica} implementation shows that $Q_{\alpha}\left(
\kappa\right)  $ for this problem can be factorized into five polynomials
$\eta_{1}\left(  \kappa\right)  =\kappa^{10}$, $\eta_{2},\eta_{3}\in
\mathbb{P}_{2}$, $\eta_{4}\in\mathbb{P}_{4}$, $\eta_{5}\in\mathbb{P}_{6}:$
$Q_{\alpha}\left(  \kappa\right)  =f_{\alpha}\eta_{1}\eta_{2}^{2}\eta_{3}%
\eta_{4}\eta_{5}^{2}$ with a positive rational function $f_{\alpha}$ for all
$\alpha\in\left]  0,1\right[  $. Clearly, the only root of $\eta_{1}$ is
$\kappa=0$ which is not relevant for critical wavenumbers $k\in\mathbb{R}%
\backslash\left\{  0\right\}  $; $\eta_{2}$ has two positive roots\footnote{It
is easy to see that, by expanding the fraction for the case $\left(  \left(
\cdot\right)  -\sqrt{\cdot}\right)  $ by $\left(  \left(  \cdot\right)
+\sqrt{\cdot}\right)  $, the numerator becomes $6\alpha^{2}\left(
4+\alpha\left(  10-8\alpha-\alpha^{2}+\alpha^{3}\right)  \right)  $, i.e.,
positive for any $\alpha\in\left]  0,1\right[  $ while the positivity of the
denominator is obvious. For the case $\left(  \left(  \cdot\right)
+\sqrt{\cdot}\right)  $ the positivity is obvious.}%
\[
\kappa_{\pm}\left(  \alpha\right)  =\frac{12}{\alpha^{2}}\frac{2+7\alpha
+3\alpha^{2}-3\alpha^{3}\pm\sqrt{3\alpha^{6}+3(5-4\alpha)\alpha^{4}%
+(37-30\alpha)\alpha^{2}+28\alpha+4}}{2+6\alpha-\alpha^{2}-\alpha^{3}},
\]
while the imaginary parts of the roots of $\eta_{3}$ are always non-zero for
$\alpha\in\left]  0,1\right[  $. We do not investigate the roots of $\eta_{4}$
and $\eta_{5}$. It follows that the conforming Galerkin discretization of the
transverse magnetic problem in 2D by Raviart-Thomas finite elements
$\mathbf{RT}_{1}\left(  {\mathcal{T}}_{2}\right)  $ has a \textit{spurious
mode} for any $\alpha\in\left]  0,1\right[  $ and $k^{2}=\kappa_{\pm}\left(
\alpha\right)  $.

\subsection{Maxwell 3D, $\mathbf{N}_{p}^{\operatorname*{I}}$ for p=1:}

For any $\alpha\in\left]  0,1\right[  $ and $p=1$ symbolic manipulation by the
\textsc{Mathematica} implementation shows that $Q_{\alpha}\left(
\kappa\right)  $ for this problem can be factorized into nine polynomials
$\eta_{1}\left(  \kappa\right)  =\kappa^{10}$, $\eta_{2}\in\mathbb{P}_{1}$,
$\eta_{3},\eta_{4}\in\mathbb{P}_{2}$, $\eta_{5}\in\mathbb{P}_{3}$, $\eta
_{6},\eta_{7},\eta_{8}\in\mathbb{P}_{4}$, $\eta_{9}\in\mathbb{P}_{8}:$
$Q_{\alpha}\left(  \kappa\right)  =\frac{\eta_{1}\eta_{2}^{2}\eta_{3}\eta
_{4}^{2}\eta_{5}^{2}\eta_{6}\eta_{7}\eta_{8}\eta_{9}^{2}}{f_{\alpha}}$ with a
positive polynomial $f_{\alpha}$ for all $\alpha\in\left]  0,1\right[  $.
Clearly, the only root of $\eta_{1}$ is $\kappa=0$ which is not relevant for
critical wavenumbers $k\in\mathbb{R}\backslash\left\{  0\right\}  $; $\eta
_{2}$ has a positive root%
\[
\kappa\left(  \alpha\right)  =20\frac{2-\alpha}{1+\alpha\left(  1-\alpha
\right)  },
\]
while the imaginary parts of the roots of $\eta_{3},\eta_{4}$ are always
non-zero for $\alpha\in\left]  0,1\right[  $. We do not investigate the roots
of $\eta_{6},\ldots,\eta_{9}$.

The polynomial $\eta_{5}$ requires a more subtle analysis. Define coefficients%
\begin{align}
c_{0}\left(  \alpha\right)   &  :=1024000\left(  2\alpha^{2}+\alpha+2\right)
,\label{defcalpha}\\
c_{1}\left(  \alpha\right)   &  :=-1600\left(  \alpha^{3}\left(
60-19\alpha\right)  +47\alpha^{2}+136\alpha+36\right)  ,\nonumber\\
c_{2}\left(  \alpha\right)   &  :=20\left(  \alpha^{3}\left(  222-35\alpha
^{3}\right)  +\left(  16-15\alpha^{5}\right)  +12\alpha^{4}+188\alpha
^{2}+112\alpha\right) \nonumber\\
c_{3}\left(  \alpha\right)   &  :=-\left(  \alpha^{4}\left(  26-21\alpha
\right)  +\alpha^{2}\left(  41-31\alpha\right)  +52\alpha+8\right)  \alpha^{2}
\label{defcalphaend}%
\end{align}
and the polynomial%
\begin{equation}
p_{\alpha}\left(  x\right)  =\sum_{\ell=0}^{3}c_{\ell}\left(  \alpha\right)
x^{\ell}. \label{defpalpha}%
\end{equation}

All three roots of the polynomial $p_{\alpha}$ are real, positive, and
distinct. Since the proof is somehow tedious but elementary we have postponed
it to Appendix \ref{ApTech}. Here we depict the graph of the discriminant:%
\[
\Delta:=-18c_{0}c_{1}c_{2}c_{3}+4c_{2}^{3}c_{0}-c_{1}^{2}c_{2}^{2}+4c_{3}%
c_{1}^{3}+27c_{3}^{2}c_{0}^{2}%
\]
motivated by the following reasoning: It holds (see, e.g., \cite[Chap. IX,
p.122]{turnball_theory_eqs})%
\[
\Delta\left(  \alpha\right)  <0\iff\text{all three roots of }p_{\alpha}\text{
are distinct and real.}%
\]
Note that the coefficients $c_{\ell}\left(  \alpha\right)  $ in $p_{\alpha
}\left(  x\right)  $ satisfy sign properties which directly follow from their
representations (\ref{defcalpha})-(\ref{defcalphaend}): for $\alpha\in\left[
0,1\right]  $ it holds:%
\[
c_{0}\left(  \alpha\right)  >0,\quad c_{1}\left(  \alpha\right)  <0,\quad
c_{2}\left(  \alpha\right)  >0,\quad c_{3}\left(  \alpha\right)  <0
\]
and $p_{\alpha}\left(  x\right)  >0$ for all $x\leq0$ follows. Consequently:
if $\Delta<0$ in $\left]  0,1\right[  $ then all three roots of $p_{\alpha}$
are distinct, real, and positive.\ The graph of $\Delta$ is shown in Figure
\ref{FigDiscr}.%
%TCIMACRO{\FRAME{ftbpFU}{3in}{2.0003in}{0pt}{\Qcb{Discriminant $\Delta$ of the
%cubic polynomial $\eta_{5}$ as a function of $\alpha\in\left[  0,1\right]  $%
%.}}{\Qlb{FigDiscr}}{discriminantswp.eps}%
%maintain-aspect-ratio TRUE;  display "USEDEF";  valid_file "F";  width 3in;
%height 2.0003in;  depth 0pt;  original-width 0pt;  original-height 0pt;
%cropleft "0";  croptop "1";  cropright "1";  cropbottom "0";
%filename 'discriminantswp.eps';file-properties "XNPEU";}} }%
%BeginExpansion
\begin{figure}[ptb]%
\centering
\includegraphics[
height=2.0003in,
width=3in
]%
{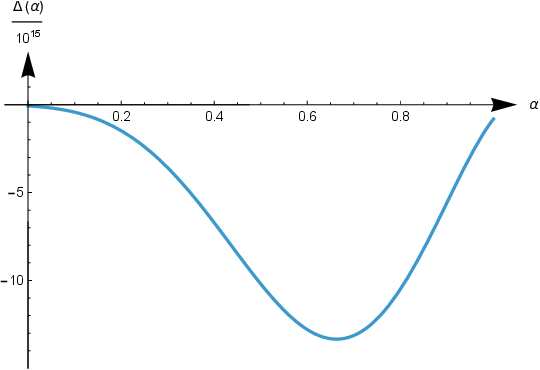}%
\caption{Discriminant $\Delta$ of the cubic polynomial $\eta_{5}$ as a
function of $\alpha\in\left[  0,1\right]  $.}%
\label{FigDiscr}%
\end{figure}
%EndExpansion

{In summary, t}he conforming Galerkin discretization by the N\'{e}d\'{e}lec
space $\mathbf{N}_{1}^{\operatorname*{I}}\left(  {\mathcal{T}}_{3}\right)  $
has eight spurious wavenumbers%
\[
k\left(  \alpha\right)  \in\left\{  \pm\sqrt{20\frac{2-\alpha}{1+\alpha\left(
1-\alpha\right)  }},\pm\sqrt{x_{1}\left(  \alpha\right)  },\pm\sqrt
{x_{2}\left(  \alpha\right)  },\pm\sqrt{x_{3}\left(  \alpha\right)  }\right\}
,
\]
where $x_{m}\left(  \alpha\right)  $, $1\leq m\leq3$, are the three distinct
and positive roots of the equation
\begin{equation}
p_{\alpha}\left(  x\right)  =0. \label{cubiceq}%
\end{equation}

\section{Conclusion\label{SecConcl}}

In this paper we considered three common models in two and three spatial
dimensions for high-frequency scattering problems, all of which are well posed
on the continuous level. We constructed simplicial meshes for popular,
low-order, conforming Galerkin finite element discretizations for these
variational problems such that the discrete system is singular for certain
discrete values of $k$. These meshes are highly symmetric and for all
constructed examples, the number of degrees of freedom is relatively low on
the domain boundary.

Even for these specially constructed meshes the system matrix is singular only
{for} finitely many spurious frequencies $k>0$ indicating that these cases are
very rare events.

Next, we comment on the resolution condition applied to periodic copies of the
Bernkopf meshes. Let $N\in\mathbb{N}$ and $h=1/\left(  N+1\right)  $. Let
$\mathcal{T}_{3}^{+}$ denote some conforming tetrahedral mesh of $\left(
-1,1\right)  ^{3}$ which a) contains $\mathcal{T}_{3}$ as a submesh and b) the
four congruent triangles, which arise by inserting the two diagonals into each
facet of the cube $\left(  -1,1\right)  ^{3}$, are facets of some tetrahedra
in $\mathcal{T}_{3}^{+}$. For a uniform notation we set $\mathcal{T}_{2}%
^{+}:=\mathcal{T}_{2}$.

A scaling of the mesh $\mathcal{T}_{d}^{+}$, $d=2,3$, by a factor $h/2$ leads
to a mesh $\mathcal{T}_{d,h}^{+}$ of $\Omega_{h}:=\left(  -\frac{h}{2}%
,\frac{h}{2}\right)  ^{d}$ which contains the scaled version $\mathcal{T}%
_{d,h}$ of $\mathcal{T}_{d}$ as a submesh. The translation $\Phi_{%
%TCIMACRO{\TeXButton{boldmue}{\boldsymbol{\mu}}}%
%BeginExpansion
\boldsymbol{\mu}%
%EndExpansion
}:\Omega_{h}\rightarrow\Omega_{h}\left(
%TCIMACRO{\TeXButton{boldmue}{\boldsymbol{\mu}}}%
%BeginExpansion
\boldsymbol{\mu}%
%EndExpansion
\right)  $ is given for
\[%
%TCIMACRO{\TeXButton{boldmue}{\boldsymbol{\mu}}}%
%BeginExpansion
\boldsymbol{\mu}%
%EndExpansion
\in\mathcal{J}_{N}:=\left\{
%TCIMACRO{\TeXButton{boldnue}{\boldsymbol{\nu}}}%
%BeginExpansion
\boldsymbol{\nu}%
%EndExpansion
=\left(  \nu_{j}\right)  _{j=1}^{d}\in\mathbb{N}_{0}^{d}\mid\left\Vert
%TCIMACRO{\TeXButton{boldmue}{\boldsymbol{\mu}}}%
%BeginExpansion
\boldsymbol{\mu}%
%EndExpansion
\right\Vert _{\max}\leq N\right\}  \quad\text{and\quad}\mathbf{1}=\left(
1\right)  _{i=1}^{d}%
\]
by%
\[
\Phi_{%
%TCIMACRO{\TeXButton{boldmue}{\boldsymbol{\mu}}}%
%BeginExpansion
\boldsymbol{\mu}%
%EndExpansion
}\left(  \mathbf{x}\right)  =\mathbf{x}+h\left(
%TCIMACRO{\TeXButton{boldmue}{\boldsymbol{\mu}}}%
%BeginExpansion
\boldsymbol{\mu}%
%EndExpansion
+\frac{1}{2}\mathbf{1}\right)  .
\]
The simplices in $\Omega_{h}\left(
%TCIMACRO{\TeXButton{boldmue}{\boldsymbol{\mu}}}%
%BeginExpansion
\boldsymbol{\mu}%
%EndExpansion
\right)  $, $%
%TCIMACRO{\TeXButton{boldmue}{\boldsymbol{\mu}}}%
%BeginExpansion
\boldsymbol{\mu}%
%EndExpansion
\in\mathcal{J}_{N}$, define a conforming simplicial mesh for $\left(
0,1\right)  ^{d}$ where each cell $\Omega_{h}\left(
%TCIMACRO{\TeXButton{boldmue}{\boldsymbol{\mu}}}%
%BeginExpansion
\boldsymbol{\mu}%
%EndExpansion
\right)  $ contains a copy of the scaled structured mesh $\mathcal{T}_{d,h}$.
By a scaling argument (which is detailed in \cite[Lem. 3.3]%
{bernkopf2021solvability}) one obtains that there are non-trivial homogeneous
solutions of the discrete problem which are localized in one of the cells if%
\[
kh=k\left(  \alpha\right)
\]
for the spurious wavenumbers $k\left(  \alpha\right)  $ defined for the
Galerkin discretization of our model problems on the constructed meshes. On
the other hand, this observation can also be formulated in a reversed way:
since $k\left(  \alpha\right)  $ is independent of $h$ a condition
\[
kh<k\left(  \alpha\right)
\]
ensures that no non-trivial discrete homogenous solutions exist which are
localizing in one cell. Such a condition is substantially weaker than a
typical condition \textquotedblleft for some fixed $\nu>1$, the product
$k^{\nu}h$ must be sufficiently small\textquotedblright\ as it appears, e.g.,
with $\nu=2$ in (\ref{adapproxMW2D}).

\appendix

\section{Analysis of the roots of $p_{\alpha}$ in (\ref{defpalpha}%
)\label{ApTech}}

In this appendix we prove the following lemma.

\begin{lemma}
\label{Lemrootprop}The cubic polynomial $p_{\alpha}$ in (\ref{defpalpha}) has
three distinct positive roots for any $\alpha\in\left[  0,1\right]  $.
\end{lemma}

%

%TCIMACRO{\TeXButton{Proof}{\proof}}%
%BeginExpansion
\proof
%EndExpansion
The coefficients $c_{m}\left(  \alpha\right)  $ in $p_{\alpha}\left(
x\right)  =\sum_{\ell=0}^{3}c_{\ell}\left(  \alpha\right)  x^{\alpha}$ satisfy
the following sign properties which directly follow from their representations
(\ref{defcalpha})-(\ref{defcalphaend}): for $\alpha\in\left[  0,1\right]  $ it
holds:%
\[
c_{0}\left(  \alpha\right)  >0,\quad c_{1}\left(  \alpha\right)  <0,\quad
c_{2}\left(  \alpha\right)  >0,\quad c_{3}\left(  \alpha\right)  <0.
\]

This implies
\begin{align}
\operatorname*{sign}p_{\alpha}\left(  0\right)   &  =\operatorname*{sign}%
c_{0}\left(  \alpha\right)  =1,\label{sign1}\\
\operatorname*{sign}p_{\alpha}\left(  +\infty\right)   &
=\operatorname*{sign}c_{3}\left(  \alpha\right)  =-1. \label{sign2}%
\end{align}

Next we prove that
\begin{equation}
\operatorname*{sign}p_{\alpha}\left(  50\right)  =-1. \label{sign3}%
\end{equation}

\textbf{Case 1. }$\alpha\in\left[  \frac{1}{2},1\right]  $ and we set
$b=1-\alpha\in\left[  0,\frac{1}{2}\right]  $. For $x=50$ it holds $p_{\alpha
}\left(  50\right)  =-\frac{1000}{\alpha^{7}}w\left(  b\right)  $ for%

\begin{equation}
w\left(  b\right)  =13375b^{5}\left(  1-b\right)  +2625b^{7}+b^{2}\tilde
{w}_{3}\left(  b\right)  +500b+55 \label{wofb}%
\end{equation}
with%
\[
\tilde{w}_{3}\left(  b\right)  :=14\,875b^{3}-29495b^{2}+10905b+567.
\]
Next we show $\tilde{w}_{3}\left(  b\right)  \geq0$ in the range $b\in\left[
0,\frac{1}{2}\right]  $, i.e., $\alpha\in\left[  \frac{1}{2},1\right]  $. A
straightforward discussion of $\tilde{w}_{3}$ shows that it has a local
minimum at $\frac{347}{525}+\frac{2}{8925}\sqrt{3833\,194}>1$, i.e., outside
the interval $\left[  0,\frac{1}{2}\right]  $. Hence, $\tilde{w}_{3}$ has no
local minimum in this interval and%
\[
\min_{0\leq b\leq\frac{1}{2}}\tilde{w}_{3}\left(  b\right)  =\min\left\{
\tilde{w}_{3}\left(  0\right)  ,\tilde{w}_{3}\left(  \frac{1}{2}\right)
\right\}  =\frac{4041}{8}>0.
\]
This directly implies $w\left(  b\right)  \geq55$ and, in turn, $p_{\alpha
}\left(  50\right)  <0$ for $\alpha\in\left[  \frac{1}{2},1\right]  $.

\textbf{Case 2.} $\alpha\in\left[  0,\frac{1}{2}\right]  $. We write
$p_{\alpha}\left(  50\right)  =-\frac{1000}{\alpha^{7}}v\left(  \alpha\right)
$ for%

\begin{equation}
v\left(  \alpha\right)  :=2625\alpha^{6}\left(  1-\alpha\right)  +\alpha
^{4}v_{2}\left(  \alpha\right)  +\alpha v_{3}\left(  \alpha\right)  +32
\label{defvalpha}%
\end{equation}
with%
\[
v_{2}\left(  \alpha\right)  :=2375\alpha^{2}-3125\alpha+3005\quad
\text{and\quad}v_{3}\left(  \alpha\right)  :=200\alpha^{2}-6688\alpha+4256.
\]
Again, a straightforward discussion yields, that $v_{2}$ has a local minimum
at $\frac{25}{38}$ with value $\frac{150\,255}{76}>0$ and hence is positive
for all $\alpha\in\mathbb{R}$. The function $v_{3}$ has a local minimum at
$\frac{418}{25}>1$. Hence there is no local minimum in $\left[  0,\frac{1}%
{2}\right]  $ and%
\[
\min_{\alpha\in\left[  0,\frac{1}{2}\right]  }v_{3}\left(  \alpha\right)
=\min\left\{  v_{3}\left(  0\right)  ,v_{3}\left(  \frac{1}{2}\right)
\right\}  =962.
\]
We use these positivity estimates in (\ref{defvalpha}) to get%
\[
\min_{\alpha\in\left[  0,\frac{1}{2}\right]  }v\left(  \alpha\right)  \geq32.
\]
In turn we have proved that
\[
\forall\alpha\in\left[  0,1\right]  \text{\qquad}p_{\alpha}\left(  50\right)
<0.
\]

Finally we prove
\begin{equation}
\forall\alpha\in\left[  0,1\right]  \text{\quad}\exists y_{\alpha}%
>50\quad\operatorname*{sign}p_{\alpha}\left(  y_{\alpha}\right)  =1
\label{sign4}%
\end{equation}
and distinguish between two cases:

\textbf{Case 1. }For $\alpha\in\left[  0,\frac{1}{2}\right]  $ we prove that
$p_{\alpha}\left(  150\right)  >0$. It holds $p_{\alpha}\left(  150\right)
=\frac{1000}{\alpha^{7}}z\left(  \alpha\right)  $ with%
\begin{equation}
z\left(  \alpha\right)  :=\alpha^{5}z_{2}\left(  \alpha\right)  +\alpha
z_{3}\left(  \alpha\right)  +608 \label{defzalpha}%
\end{equation}
with%
\begin{align*}
z_{2}\left(  \alpha\right)   &  :=70875\alpha^{2}-103500\alpha+97875,\\
z_{3}\left(  \alpha\right)   &  :=-128415\alpha^{3}-90000\alpha^{2}%
+48368\alpha+18784.
\end{align*}
A straightforward analysis yields that $z_{2}$ has a local minimum at
$\alpha=\frac{46}{63}$ with positive value $\frac{420\,625}{7}$ and hence
$z_{2}\left(  \alpha\right)  >0$ for all $\alpha\in\mathbb{R}$. A
straightforward discussion of the function $z_{3}$ shows that it has a local
minimum at $\alpha=-\frac{2000}{8561}-\frac{4}{128\,415}\sqrt{185\,649\,515}%
<0$. Consequently if has no local minimum in $\left[  0,\frac{1}{2}\right]  $
and%
\[
\min_{\alpha\in\left[  0,1/2\right]  }z_{3}\left(  \alpha\right)
=\min\left\{  z_{3}\left(  0\right)  ,z_{3}\left(  \frac{1}{2}\right)
\right\}  =\frac{35\,329}{8}>0.
\]
From (\ref{defzalpha}) it follows $z\left(  \alpha\right)  \geq608$ for all
$\alpha\in\left[  0,1/2\right]  $ and in turn $p_{\alpha}\left(  150\right)
>0$ in this range.

\textbf{Case 2.} For $\alpha\in\left[  \frac{1}{2},1\right]  $ we prove that
$p_{\alpha}\left(  y_{\alpha}\right)  >0$ for $y_{\alpha}=105-50\alpha$. Note
that $y_{\alpha}>50$ in this range. We obtain $p_{\alpha}\left(
105-50\alpha\right)  =\frac{125}{\alpha^{7}}\zeta\left(  b\right)  $ for
$b=1-\alpha\in\left[  0,\frac{1}{2}\right]  $ and%
\begin{equation}
\zeta\left(  b\right)  =100b^{9}\zeta_{1}\left(  b\right)  +b^{5}\zeta
_{3}\left(  b\right)  +b\xi_{3}\left(  b\right)  +95 \label{formulazetab}%
\end{equation}
with%
\begin{align*}
\zeta_{1}\left(  b\right)   &  :=517-210b,\\
\zeta_{3}\left(  b\right)   &  :=12\,494+165\,531b-113\,321b^{2}-6930b^{3},\\
\xi_{3}\left(  b\right)   &  :=4588+42\,097b+14\,067b^{2}-153\,097b^{3}.
\end{align*}
Clearly, the first factor $\zeta_{1}$ is positive for $b\in\left[  0,\frac
{1}{2}\right]  $. The function $\zeta_{3}\left(  b\right)  $ has a local
minimum at $\alpha=-\frac{1}{20\,790}\sqrt{16283\,038\,531}-\frac
{113\,321}{20\,790}<0$. Hence, there is no local minimum in $\left[
0,\frac{1}{2}\right]  $ and%
\[
\min_{b\in\left[  0,1/2\right]  }\zeta_{3}\left(  b\right)  =\min\left\{
\zeta_{3}\left(  0\right)  ,\zeta_{3}\left(  \frac{1}{2}\right)  \right\}
=12\,494.
\]
The function $\xi_{3}$ has a local minimum at $\frac{4689}{153\,097}-\frac
{2}{459\,291}\sqrt{4883\,163\,429}<0$. Again there is no local minimum in
$\left[  0,\frac{1}{2}\right]  $ and%
\[
\min_{b\in\left[  0,1/2\right]  }\xi_{3}\left(  b\right)  =\min\left\{
\xi_{3}\left(  0\right)  ,\xi_{3}\left(  \frac{1}{2}\right)  \right\}  =4588.
\]
From (\ref{formulazetab}) it follows that $\min_{0\leq b\leq1/2}\zeta\left(
b\right)  \geq95$ and in turn $p_{\alpha}\left(  105-50\alpha\right)  >0$ for
any $\alpha\in\left[  0,1/2\right]  $.

By using the meanvalue theorem the sign properties (\ref{sign1}),
(\ref{sign2}), (\ref{sign3}), (\ref{sign4}) of $p_{\alpha}$ imply that all
three roots of $p_{\alpha}$ are real, positive and distinct.%
%TCIMACRO{\TeXButton{End Proof}{\endproof}}%
%BeginExpansion
\endproof
%EndExpansion

\bibliographystyle{abbrv}
\bibliography{acompat,nlailu,theo}

\end{document}